\documentclass[11pt,leqno,oneside,letterpaper]{amsart}

\usepackage{amssymb,amsmath,amsfonts,latexsym, amscd, amsthm, url} 
\usepackage[letterpaper,left=1.75truein, top=1truein]{geometry}
\usepackage{color}
\usepackage[colorlinks,linkcolor=blue,citecolor=blue]{hyperref}

\usepackage[all]{xy}
\usepackage{caption,subcaption,graphicx,epsfig}
\usepackage{tikz-cd,calc}
\usepackage{epsfig}
\usepackage{graphicx}
\usepackage{tikz-cd}
\usepackage{array}
\usepackage{float}
\usepackage{mathtools}
\usetikzlibrary{knots}
\usetikzlibrary{decorations.markings}
\definecolor{light-gray}{gray}{0.92}
\definecolor{ultra-light-gray}{gray}{0.97}

\def\s{\sigma}
\def\r{\rho}

\def\a{\alpha}

\def\Z{\mathbb{Z}}

\def\S^1{\mathbb{S}^1}
\def\N{\mathbb{N}} 
\def\PSL{\mathrm{PSL}}
\def\S^3{\mathbb{S}^3}
\def\S^2{\mathbb{S}^2}
\def\D^2{\mathbb{D}^2}
\def\P^2{\mathbb{P}^2}

\graphicspath{{images/}}

\newtheorem{theorem}{Theorem}[section]
\newtheorem{lemma}[theorem]{Lemma}
\newtheorem{proposition}[theorem]{Proposition}

\newtheorem{corollary}[theorem]{Corollary} 

\theoremstyle{definition}
\newtheorem{remark}[theorem]{Remark}

\newtheorem{question}[theorem]{Question} 
 
\newtheorem{construction}[theorem]{Construction} 
\newtheorem{conjecture}[theorem]{Conjecture}

\newtheoremstyle{cases}
  {12pt plus 6 pt}
  {2pt}
  {\bfseries}   
  {}
  {\bfseries}
  {.}
  {.5em}
  {}
\theoremstyle{cases}

\numberwithin{subcase}{case} \numberwithin{subsubcase}{subcase}
\numberwithin{equation}{subsection}

\def\s{\sigma}

\renewcommand{\mod}{\operatorname{mod}}

\begin{document}

\title{Circular orderability of 3-manifold groups}

\author[Idrissa Ba]{Idrissa Ba}
\address{Department of Mathematics\\
University of Manitoba \\
Winnipeg \\
MB Canada R3T 2N2} \email{ba162006@yahoo.fr}
\urladdr{ https://server.math.umanitoba.ca/~idrissa/}

\author[Adam Clay]{Adam Clay}
\address{Department of Mathematics\\
University of Manitoba \\
Winnipeg \\
MB Canada R3T 2N2} \email{Adam.Clay@umanitoba.ca}
\urladdr{http://server.math.umanitoba.ca/~claya/}

  \thanks{Idrissa Ba was supported by a University of Manitoba postdoctoral fellowship.}
  \thanks{Adam Clay was supported by NSERC grant RGPIN-2020-05343.}

\begin{abstract}
This paper initiates the study of circular orderability of $3$-manifold groups, motivated by the L-space conjecture.  We show that a compact, connected, $\P^2$-irreducible $3$-manifold has a circularly orderable fundamental group if and only if there exists a finite cyclic cover with left-orderable fundamental group, which naturally leads to a ``circular orderability version" of the L-space conjecture.  We also show that the fundamental groups of almost all graph manifolds are circularly orderable, and contrast the behaviour of circularly orderability and left-orderability with respect to the operations of Dehn surgery and taking cyclic branched covers.  
\end{abstract}

\maketitle

\section{introduction}

For an irreducible, rational homology $3$-sphere $M$, the L-space conjecture posits a relationship between the properties of $M$ admitting a coorientable taut foliation, $M$ being \textit{not} an L-space, and $M$ having a left-orderable fundamental group (see Conjecture \ref{lspace}).   While this conjecture is known to hold for some classes of manifolds, for example graph manifolds, new techniques are needed to tackle more general classes of manifolds, or indeed, to tackle the conjecture in full generality.

With this conjecture in mind, several of the most successful techniques developed in recent years to tackle left-orderability of $\pi_1(M)$ have all shared a common theme---they all begin with an action on the circle, and use cohomological techniques to pass to an action on the real line.   For instance, in studying manifolds arising from Dehn surgery on a knot $K$ in $\mathbb{S}^3$, a common technique is to study one-parameter families of representations $\rho_t : \pi_1(\mathbb{S}^3 \setminus K) \rightarrow \mathrm{PSL}(2, \mathbb{R})$ that are built so as to provide representations that factor through the quotient groups $\pi_1(\mathbb{S}^3_{p/q}(K))$ for certain values of $p/q \in \mathbb{Q} \cup \{ \infty\}$.  Controlling the Euler classes of these representations allows one to construct lifts $\widetilde{\rho_t} : \pi_1(\mathbb{S}^3_{p/q}(K)) \rightarrow \widetilde{\mathrm{PSL}}(2, \mathbb{R})$, and these lifts show that $\pi_1(\mathbb{S}^3_{p/q}(K))$ is left-orderable since they have left-orderable image \cite{BGW,HT1,HT2,MT,CuD,Ga}.   This technique has also been used to study left-orderability of cyclic branched covers of knots \cite{Hu, Tr, Tu, GL}.

In a similar vein, if one starts with an irreducible rational homology $3$-sphere $M$ admitting a coorientable taut foliation $\mathcal{F}$, Thurston's universal circle construction yields a representation $\rho_{univ} : \pi_1(M) \rightarrow \mathrm{Homeo}_+(\mathbb{S}^1)$.  With appropriate restrictions on $\mathcal{F}$, one can control the Euler class of $\rho_{univ}$ and guarantee the existence of a lift $\widetilde{\rho_{univ}} : \pi_1(M) \rightarrow \widetilde{\mathrm{Homeo}_+}(\mathbb{S}^1)$, again yielding left-orderability of $\pi_1(M)$ for similar reasons \cite{CaD,BH}.

Motivated by the utility of actions on the circle by homeomorphisms in addressing the L-space conjecture, this work is a first step in directly addressing the question of when the fundamental group of a $3$-manifold acts on the circle by homeomorphisms---though we take an algebraic approach to the problem.  Just as the existence of left-ordering of $\pi_1(M)$ captures whether or not there is an embedding $\rho : \pi_1(M) \rightarrow \mathrm{Homeo}_+(\mathbb{R})$, we approach the problem by studying the existence of an orientation cocycle $c: \pi_1(M)^3 \rightarrow \{ \pm 1, 0\}$ that is compatible with the group operation, called a circular ordering of $\pi_1(M)$.  The existence of such a map determines whether or not there exists an embedding $\rho : \pi_1(M) \rightarrow \mathrm{Homeo}_+(\mathbb{S}^1)$, analogous to the case of left-orderings.  It should be noted that throughout the manuscript, we adopt the convention that the trivial group is left-orderable.   We show:

\begin{theorem}
\label{LO_covers}
Suppose that $M$ is a compact, connected, $\P^2$-irreducible $3$-manifold.  Then $\pi_1(M)$ is circularly orderable if and only if $M$ admits a finite cyclic cover with left-orderable fundamental group.
\end{theorem}

Our contribution here is not the existence of a finite-index left-orderable subgroup, as this fact already appears implicitly in the literature in \cite{CaD}, but that there is a normal, left-orderable subgroup that yields a finite cyclic group upon passing to the quotient.  This motivates an obvious ``circular orderability" version of the L-space conjecture (see Conjecture \ref{circular_lspace}), which mirrors the usual L-space conjecture up to finite cyclic covers.

This theorem is in fact a special case of a new algebraic result. Associated to every circular ordering $c$ of $G$ is a cohomology class $[f_c] \in H^2(G; \mathbb{Z})$, called the Euler class of the circular ordering.  It happens that when a group $G$ admits a circular ordering whose Euler class has order $k$ in $H^2(G; \mathbb{Z})$, then $G$ admits a left-orderable normal subgroup $N$ such that $G/N \cong \mathbb{Z}/k\mathbb{Z}$, see Theorem \ref{eulerclassorder}.

From here we begin an exploration of exactly which fundamental groups admit circular orderings.  We first tackle the case of Seifert fibred manifolds, providing the details of a claim appearing in \cite{Cal}.  Note that circularly orderability of finite groups is well understood (a finite group is circularly orderable if and only if it is cyclic, see Proposition \ref{finite case}), and so we focus on infinite fundamental groups. If $G$ is a group with circular ordering $c$, we use $\mathrm{rot}_c(g)$ to denote the rotation number of $g \in G$, see Section \ref{intro_sect}.  

\begin{theorem}
\label{SFCO}
Let $M$ be a compact, connected Seifert fibred space and let $h$ denote the class of a regular fibre in $\pi_1(M)$.  
\begin{enumerate}
\item If $\pi_1(M)$ is infinite, then there exists a circular ordering $c$ of $\pi_1(M)$ such that $\mathrm{rot}_c(h) = 0$, in particular, $\pi_1(M)$ is circularly orderable whenever it is infinite. 
\item If $\pi_1(M)$ is infinite, $M$ is orientable and has nonorientable base orbifold,
then every circular ordering $c$ of $\pi_1(M)$ satisfies $\mathrm{rot}_c(h) \in \{ 0, 1/2 \}$.
\item If $\pi_1(M)$ is left-orderable and $M$ is orientable and has base orbifold $\mathbb{S}^2(\alpha_1, \dots, \alpha_n)$ where $n\geq 3$,
then for every $p \in \mathbb{N}_{>0}$ there exists a circular ordering $c$ of $\pi_1(M)$ such that $\mathrm{rot}_c(h) = 1/p$. 
\item If $M$ is orientable and has no exceptional fibres, then for every $r \in \mathbb{R}/\mathbb{Z}$ there exists a circular ordering $c$ of $\pi_1(M)$ such that $\mathrm{rot}_c(h) = r$.
\end{enumerate}
\end{theorem}

This leads naturally to the study of graph manifolds, where we show that an analogous fact holds.  

\begin{theorem}
\label{special graph case}
Suppose that $W$ is a graph manifold whose JSJ decomposition has Seifert fibred pieces $M_1, \dots, M_n$.  Further suppose that for each $1 \leq i \leq n$, if $\partial M_i$ is a single torus boundary component then there is no slope $\alpha \in H_1(\partial M_i ; \mathbb{Z})/\{ \pm 1\}$ such that $\pi_1(M_i(\alpha))$ is finite.  Then if $\pi_1(W)$ is infinite, it is circularly orderable.
\end{theorem}

If $W$  is not a rational homology sphere then the first Betti number $b_1(W)$ is positive, so $\pi_1(W)$ is left-orderable by \cite[Theorem 3.2]{BRW}. On the other hand if $W$ is a rational homology sphere, then Theorem \ref{special graph case} is in fact a special case of a stronger, more technical result, see Theorem \ref{graphCO} and Corollary \ref{2piecescase}.  We conjecture that with appropriate generalizations of the techniques developed here, one can prove that the fundamental group of a graph manifold is circularly orderable whenever it is infinite.  See Conjecture \ref{graph conjecture} and preceding discussion for details.

Our approach to this proof is to mirror the technique of ``slope detection'' developed in \cite{BC} for the case of left-orderings of fundamental groups of graph manifolds; we develop a result in the case of circular orderings that is analogous to the main tool of \cite{CLW}, see Theorem \ref{CO slope}. This tool provides sufficient conditions that a manifold $W = M_1 \cup_{\phi} M_2$ have circularly orderable fundamental group, by requiring that the gluing map $\phi$ identify slopes on $M_1$ and $M_2$ whose fillings yield fundamental groups admitting compatible circular orderings.  Using this technique, it turns out that in many cases it is sufficient to study fillings along rational longitudes to conclude that $W = M_1 \cup_{\phi} M_2$ has circularly orderable fundamental group, see Proposition \ref{rat_long_CO}.  

We also deal with several notable cases not covered by Theorem \ref{graphCO} or Theorem \ref{special graph case}, for instance we also show:

\begin{theorem}\label{thmSol}
The fundamental group of a compact, connected Sol manifold is circularly orderable. 
\end{theorem}

We close with a discussion of circularly orderability of fundamental groups of hyperbolic 3-manifolds. There is a well known example of a hyperbolic 3-manifold whose fundamental group is not circularly orderable, which is the Weeks manifold \cite[Theorem 9]{CaD}. Therefore, we cannot expect the fundamental groups of hyperbolic 3-manifolds to be circularly orderable whenever they are infinite, as in the case of Seifert fibred manifolds.  

Two approaches to the question of left-orderability of fundamental groups of hyperbolic $3$-manifolds that have enjoyed success are via cyclic branched covers, and the other via Dehn surgery.  In both of these cases, advancements in Heegaard-Floer techniques have provided guidance as to the expected behaviour of left-orderability with respect to these constructions.  Over the course of several examples, including several infinite families of hyperbolic $3$-manifolds having circularly orderable but non-left-orderable fundmental groups, we find that none of the behaviour exhibited by left-orderability with respect to these familiar topological constructions is shared with circular orderability.   

For example, it is suspected that if the $n$-fold cyclic branched cover of a knot in $\mathbb{S}^3$ has left-orderable fundamental group, then so does the $m$-fold cyclic branched cover for all $m \geq n$.  This does not hold for circular orderability, see Propositions \ref{trefoil prop} and \ref{52 prop}.  Similarly it is conjectured that the double branched cover of a quasi-alternating knot always has non-left-orderable fundamental group, but our examples show that there is no apparent relationship when left-orderability is replaced with circular orderability:  There exist alternating links (more generally quasi-alternating links) whose double branched covers have non-circularly orderable fundamental groups (see Sections \ref{sec4}), while other alternating (or quasi-alternating) links yield double branched covers with circularly orderable fundamental groups.  Similar observations hold for the behaviour of circular orderability with respect to Dehn surgery on a knot in $\mathbb{S}^3$.

We organize this paper as follows: Section \ref{intro_sect} contains background and results relating to circular orderability and left-orderability of groups in general. In Section \ref{3-manifold facts} we relate these facts to $3$-manifold fundamental groups, discuss the L-space conjecture and prove Theorem \ref{LO_covers}.  In Section \ref{sec1} we introduce our tools that are analogous to slope detection by left-orderings, and in Section \ref{rat_long} we show how these results can be applied to fillings along rational longitudes.  In Section \ref{seifert and graph} we study circularly orderability of the fundamental groups of  Seifert fibred manifolds and graph manifolds.  Finally, in Section \ref{sec3} we discuss circularly orderability of the fundamental groups of manifolds arising as the cyclic branched covers of links, and manifolds arising from Dehn surgery.

\section{Left and circular orderability}
\label{intro_sect}

A strict total order $<$ on a group $G$ is said to be a {\it left-ordering} if for every $f, g, h \in G$, if $g<h$ then 
$fg<fh$. A group $G$ is called {\it left-orderable} if it admits a 
left-ordering.  Every left-ordering of $G$ determines a subset $P = \{ g \in G \mid g > id \}$ called the \textit{positive cone} of the ordering, it satisfies (i) $P \cdot P \subset P$, and (ii) $P \sqcup P^{-1} = G \setminus \{id \}$.  Conversely any subset $P \subset G$ satisfying (i) and (ii) determines a left-ordering of $G$ via the prescription 
\[ g<h \Longleftrightarrow g^{-1}h \in P.
\]

A {\it left-circular ordering} of a group $G$ is a map $c: G^3 \rightarrow \{ \pm 1, 0\}$ satisfying:
	\begin{enumerate}
		\item If $(g_1, g_2, g_3) \in G^3$ then $c(g_1, g_2, g_3) = 0$ if and only if $\{g_1, g_2, g_3\}$ are not all distinct;
		\item For all $g_1, g_2, g_3, g_4 \in G$ we have
		\[ c(g_1, g_2, g_3) - c(g_1, g_2, g_4) + c(g_1, g_3, g_4)-c(g_2, g_3, g_4) = 0;
		\]
		and
		\item For all $g, g_1, g_2, g_3 \in G$ we have 
		\[ c(g_1, g_2, g_3) = c(gg_1, gg_2, gg_3).
		\]
	\end{enumerate}
	If $G$ admits such a map, then $G$ is called {\it left-circularly orderable}.  When no confusion will arise from doing so, we will write \emph{circular ordering} for short and \emph{circularly orderable}.

Every left-orderable group is circularly orderable, for if $<$ is a left-ordering of $G$ then we may define $c:G^3 \rightarrow \{ \pm 1, 0\}$ by $c(g_1, g_2, g_3) = \mathrm{sign}(\sigma)$ when $\{g_1, g_2, g_3\}$ are distinct and $c(g_1, g_2, g_3) =0$ otherwise; here $\sigma$ is the unique permutation such that $g_{\sigma(1)} <g_{\sigma(2)} <g_{\sigma(3)}$.  When a circular ordering $c$ of a left-orderable group $G$ arises in this way, we will say that $c$ is a \emph{secret left-ordering}.

Every circular ordering $c$ of $G$ is evidently a homogeneous cocycle.  However, from each circular ordering $c$ we can define an associated inhomogeneous cocycle $f_c :G^2 \rightarrow \{0, 1\}$ by
\[
f_c(g,h) = \begin{cases}
0 &\text{if } g=id \text{ or }h = id,\\
1 &\text{if } gh = id \text{ and } g \neq id, \\
\frac 12(1 - c(id,g,gh)) &\text{otherwise.}
\end{cases}
\]
we call $[f_c] \in H^2(G; \mathbb{Z})$ the \emph{Euler class} of the circular ordering $c$.

\begin{construction}
\label{lift_construction} \cite{Zel}
Associated to $[f_c]$ is a central extension $\widetilde{G}_c$ of $G$, which is constructed by equipping the set $\mathbb{Z} \times G$ with the operation $(a, g)(b, h) = (a+b+f_c(g,h), gh)$.\footnote{This is just an application of the standard construction associating elements of $H^2(G;\mathbb{Z})$, represented by inhomogeneous $2$-cocycles, to equivalence classes of central extensions $1 \rightarrow \mathbb{Z} \rightarrow \widetilde{G} \rightarrow G \rightarrow 1$.}  The central extension $\widetilde{G}_c$ is easily seen to be left-orderable, as one can check that the set $P = \{ (a, g) \mid a \geq 0\} \setminus \{ (0, id )\}$ defines the positive cone of a left-ordering that we denote by $<_c$.  We call $\widetilde{G}_c$ the \textit{left-ordered central extension associated to the circularly ordered group $G$ with ordering $c$}.

Recall that a subset $S$ of a left-ordered group $(G,<)$ is \textit{$<$-cofinal} if, for every $g \in G$, there exist elements $s, t \in S$ such that $s<g<t$.  An element $g \in G$ is called $<$-cofinal (or simply cofinal if the ordering is understood) whenever the cyclic subgroup $\langle g \rangle$ is $<$-cofinal as a set.  The central element $(1, id) \in \widetilde{G}_c$ is positive and cofinal in the left-ordering $<_c$ of $ \widetilde{G}_c$ and is called the \textit{canonical positive, cofinal, central element of $\widetilde{G}_c$}.
\end{construction}

\begin{construction} \cite{Zel}
\label{quotient_construction}
The above construction is reversible, in a categorical sense made precise in \cite{CG}, the basic construction is as follows.  Suppose that $G$ is a left-ordered group with ordering $<$, and there is a central element $z \in G$ which is positive and $<$-cofinal.  Then the quotient $G/\langle z \rangle$ inherits a circular ordering defined as follows.  For each $g \langle z \rangle \in G/\langle z \rangle$, define the \textit{minimal representative} $\bar{g}$ to be the unique coset representative of $g \langle z \rangle$ satisfying $id \leq \bar{g} < z$.  Then define a circular ordering $c_<$ on $G/\langle z \rangle$ by
\[ c_{<}(g_1 \langle z \rangle, g_2 \langle z \rangle, g_3 \langle z \rangle) = \mathrm{sign}(\sigma),
\]
where $\sigma$ is the unique permutation satisfying $\overline{g_{\sigma(1)}} < \overline{g_{\sigma(2)}} < \overline{g_{\sigma(3)}}$.
\end{construction}

When $G$ admits a circular ordering $c$ with $[f_c] = id \in H^2(G; \mathbb{Z})$, then $G$ is left-orderable, because the left-orderable central extension $\widetilde{G}_c$ is isomorphic to $G \times \mathbb{Z}$ (though the ordering $c$ need not be a secret left ordering for this to happen).  It happens that 
$c$ is a secret left-ordering if and only if $[f_c] = id \in H^2_b(G ; \mathbb{Z})$, where $H^2_b(G ; \mathbb{Z})$ is the second bounded cohomology group (see \cite{BCG} for details).

We also recall the notion of  \textit{rotation number} of an orientation-preserving homeomorphism $f: \mathbb{S}^1 \rightarrow \mathbb{S}^1$, which is connected to a circular ordering and the lift $\widetilde{G}_c$ as follows. For an orientation-preserving homeomorphism $f : \mathbb{S}^1 \rightarrow \mathbb{S}^1$, one may choose a preimage $\widetilde{f} \in \widetilde{\mathrm{Homeo}}_+(\mathbb{S}^1)$ of $f \in \mathrm{Homeo}_+(\mathbb{S}^1)$ and define the rotation number of $f$ to be
\[ \lim_{n \to \infty} \cfrac{\widetilde{f}^n(0)}{n} \, \mod \mathbb{Z}. 
\]

We can define rotation number for an element $g$ of a circularly ordered group $(G, c)$, following \cite{BaC}, in a similar way.  Let $z=(1, id) \in \widetilde{G}_c$ denote the cofinal, central element of $ \widetilde{G}_c$ relative to the ordering $<_c$.  Choose an element $\widetilde{g} \in \widetilde{G}_c$ such that $q(\widetilde{g}) = g$, where $q:  \widetilde{G}_c \rightarrow G$ is the quotient map.  For each $n \in \mathbb{Z}$, let $a_n$ denote the unique integer such that 
\[ z^{a_n} \leq \widetilde{g}^n < z^{a_n+1},
\] 
and define
\[ \mathrm{rot}_c(g) = \lim_{n \to \infty} \cfrac{a_n}{n} \, \mod \mathbb{Z}. \]
Note that this limit always exists by Fekete's lemma, as one can check that the sequence $\{a_n \}_{n=1}^{\infty}$ is superadditive.   Is it not difficult, though rather tedious, to show that this notion of rotation number agrees with the ``traditional definition'' if one uses the circular ordering $c$ of $G$ to create a dynamical realization $\rho_c : G \rightarrow \mathrm{Homeo}_+(\mathbb{S}^1)$ such that the circular ordering $c$ of $G$ agrees with the circular ordering of the orbit of $0$ inherited from the natural circular ordering of $\mathbb{S}^1$ (see \cite[Sections 2.2 and 2.4]{CG3}).  In particular, this implies that rotation number is invariant under conjugation, and is a homomorphism from $A \rightarrow S^1$ when restricted to any abelian subgroup $ A \subset G$.  Moreover, the induced homomorphism $A/\ker(\mathrm{rot}_c) \rightarrow \mathbb{S}^1$ is order-preserving with left-ordered kernel, see \cite[Propositions 5.3 and 6.17]{Gh} and \cite[Section 2]{CMR}.

A fundamental tool in constructing circular orderings on a given group $G$ is the lexicographic construction, which we use often throughout this work.

\begin{proposition} 
\label{sesprop}
Let \[ 1 \rightarrow K \rightarrow G \stackrel{q}{\rightarrow} H \rightarrow 1
\] be a short exact sequence of groups.
\begin{enumerate}
\item If $K$ and $H$ are left orderable then $G$ is left orderable.
\item If $K$ is left-orderable and $H$ admits a circular ordering $d$, then $G$ admits a circular ordering $c$ satisfying $\mathrm{rot}_d(q(g)) = \mathrm{rot}_c(g)$ for all $g \in G$, and whose restriction to $K$ is secretly a left-ordering.
\end{enumerate}
\end{proposition}
\begin{proof}
Claim (1) is a straightforward exercise and is common in the literature. Claim (2) is less common in the literature, so we outline a lexicographic construction following (\cite[Lemma 2.2.12]{Cal}) and verify that the claimed properties of the resulting circular ordering.  Suppose $K$ is equipped with the left-ordering $<$.  Define a circular ordering $c : G \rightarrow \{ 0 , \pm 1 \}$ as follows.  Given three distinct elements $g_1, g_2, g_3 \in G$:

\noindent \textbf{Case 1.} If $q(g_i)$ are all distinct, set $c(g_1, g_2, g_3) = d(q(g_1), q(g_2), q(g_3))$.

\noindent \textbf{Case 2.} If exactly two of $\{ q(g_1), q(g_2), q(g_3) \}$ are equal, we may (by cyclically permuting the arguments and relabeling if necessary) assume that $q(g_1) = q(g_2)$, in which case we declare $c(g_1, g_2, g_3) = 1$ if $g_1^{-1} g_2>id$ and $c(g_1, g_2, g_3) = -1$ otherwise.

\noindent \textbf{Case 3.} If all of $\{ q(g_1), q(g_2), q(g_3) \}$ are equal, then declare $c(g_1, g_2, g_3) = 1$ if and only if $id < g^{-1}_1g_2<g_1^{-1}g_3$, up to cyclic permutation.

Note that if $g_1, g_2, g_3 \in K$ then we define $c(g_1, g_2, g_3)$ by appealing to Case 3, in which case we see that $c(g_1, g_2, g_3) = 1$ if and only if $g_1 <g_2 <g_3$ up to cyclic permutation.  Thus the circular ordering $c$ is a secret left-ordering upon restriction to $K$.

That $\mathrm{rot}_d(q(g)) = \mathrm{rot}_c(g)$ for all $g \in G$ is proved in \cite[Proof of Proposition 4.10]{BaC}. 
\end{proof}

 Left-orderability and circularly orderability are also well-behaved with respect to free products.
 
\begin{proposition}
\label{free product}
Let $\{G_i \}_{i \in I}$ be a family of groups.  Then 
\begin{enumerate}
\item  \cite{V}
 The free product $\ast_{i \in I} G_i$ is left-orderable if and only if each group $G_i$ is left-orderable.  Moreover, if $<_i$ is a left-ordering of $G_i$ for each $i \in I$, then there exists a left-ordering of $\ast_{i \in I} G_i$ extending the orderings $<_i$.
\item \cite[Theorem 4.2]{BS}  The free product $\ast_{i \in I} G_i$ is circularly orderable if and only if each group $G_i$ is circularly orderable. Moreover, if $c_i$ is a circular ordering of $G_i$ for each $i \in I$, then there exists a circular ordering of $\ast_{i \in I} G_i$ extending the orderings $c_i$.
\end{enumerate}
\end{proposition} 

Free products with amalgamation are much more finicky, with necessary and sufficient conditions that a free product with amalgamation be left-orderable (resp. circularly orderable) appearing in \cite{BG} (resp. \cite{CG}).  
 
Tools to obstruct circular-orderability are somewhat rarer than the tools commonly used to obstruct left-orderability.  One of the basic tools in this regard is the following fact:

\begin{proposition}
\label{finite case}
A finite group is circularly orderable if and only if it is cyclic.
\end{proposition}

For a proof from an algebraic point of view, see \cite[Proposition 2.8]{CG2}.  A circular ordering $c$ of $G$ may also give rise to a left-orderable subgroup of $G$ depending on whether or not the corresponding Euler class $[f_c] \in H^2(G ; \mathbb{Z})$ has finite order.  This allows one to obstruct circular-orderability of a group $G$ by reducing the problem to obstructing left-orderability of certain finite-index subgroups.  Calegari-Dunfield use a variant of this theorem, for instance, to show that the fundamental group of the Weeks manifold is not circularly orderable \cite{CaD}.  

\begin{theorem}
\label{eulerclassorder}
Suppose that $c$ is a circular ordering of $G$ whose Euler class $[f_c] \in H^2(G; \mathbb{Z})$ has order $k$.  Then $G$ contains a left-orderable normal subgroup $H$ such that $G/H \cong \mathbb{Z} / k \mathbb{Z}$.
\end{theorem}
\begin{proof}
 Consider the cocycle $kf_c : G^2 \rightarrow \{0, k\}$ defined by taking $k$ times the inhomogeneous cocycle $f_c$, and the corresponding central extension $\widetilde{G}_{kf_c}$ constructed as $\mathbb{Z} \times G$ with multiplication $(a,g)(b,h) = (a+b+kf_c(g,h),gh)$.    Define a map $\phi:\widetilde{G}_c \rightarrow \widetilde{G}_{kf_c}$ by $\phi(a,g) = (ka, g)$, one can check that this is an injective homomorphism.
 Since we assume $[f_c]$ has order $k$, we know that $[kf_c] = id \in H^2(G;\mathbb{Z})$, and thus there exists a map $\eta:G \rightarrow \mathbb{Z}$ satisfying $\eta(id) = 0$ and $kf_c(g,h) = \eta(g)-\eta(gh)+\eta(h)$ for all $g, h \in G$.  This map $\eta$ allows us to define an isomorphism $\psi :  \widetilde{G}_{kf_c} \rightarrow \mathbb{Z} \times G$ by $\psi(a,g) = (a +\eta(g), g)$.  
 
 Let $H$ denote the subgroup $(\psi \circ \phi )(\widetilde{G}_c) \cap (\{0\} \times G)$ of $\{0 \} \times G$, meaning that
 \[ H = \{ (0, g) \mid \exists a \in \mathbb{Z} \mbox{ such that } ka + \eta(g) =0 \}.
 \]
Then $H$ is clearly left-orderable since it is image under an injective map of a left-orderable group.   

Let $q_k : \mathbb{Z} \rightarrow \mathbb{Z}/k \mathbb{Z}$ denote the quotient map.  To the equation $\eta(gh) =  \eta(g)+\eta(h) -  kf_c(g,h)$ we apply the homomorphism $q_k$ to arrive at $(q_k \circ \eta)(gh) = (q_k \circ \eta)(g) + (q_k \circ \eta)(h)$ and conclude that $q_k \circ \eta : G \rightarrow \mathbb{Z}/k\mathbb{Z}$ is a homomorphism. Moreover, $g \in \ker(q_k \circ \eta)$ if and only if $\eta(g) \equiv 0 \mod k$, meaning that $(0,g) \in H$.  Thus the obvious isomorphism $\{0\} \times G \cong G$ carries $H$ to $\ker(q_k \circ \eta)$.  

As $q_k \circ \eta$ has image in $\mathbb{Z}/k\mathbb{Z}$, it remains to argue that $q_k \circ \eta$ is surjective.  If not, say the order of $im(q_k \circ \eta)$ is $m$, then $m$ divides $k$ and $\frac{k}{m}$ is the largest divisor of $k$ such that $\eta(g) \in \frac{k}{m}\mathbb{Z}$ for all $g \in G$.  This means that the function $\zeta(g) = \frac{m}{k}\eta(g)$ satisfies $\zeta(g) -\zeta(gh) + \zeta(h) = mf_c(g,h)$, meaning that $[f_c] \in H^2(G; \mathbb{Z})$ has order dividing $m$.   But $k \leq m$ and the order of $[f_c]$ is $k$, so this forces $m=k$ implying that $q_k \circ \eta$ is surjective.
\end{proof}

\section{Fundamental groups of $3$-manifolds and orderability}
\label{3-manifold facts}

First, we note that every compact $3$-manifold other than $\mathbb{S}^3$ admits a decomposition
\[ M \cong M_1 \# M_2 \# \dots \# M_n
\]
into prime $3$-manifolds, and as such, the fundamental group of a non-prime $3$-manifold $M$ can be expressed as a free product
\[ \pi_1(M) \cong \pi_1(M_1) * \pi_1(M_2) * \dots * \pi_1(M_n).
\]
In light of Proposition \ref{free product}(2), the question of circular orderability of fundamental groups of $3$-manifolds reduces to considering the fundamental groups of prime $3$-manifolds. In fact, since the only reducible orientable prime $3$-manifold is $\mathbb{S}^1 \times \mathbb{S}^2$, whose fundamental group is clearly circularly orderable, in the case of orientable $3$-manifolds it suffices to consider only the fundamental groups of irreducible orientable $3$-manifolds.  In the case of a nonorientable $3$-manifold $M$, the first Betti number is positive whenever $M$ is $\P^2$-irreducible, so $M$ has circularly orderable fundamental group \cite[Theorem 1.1]{BRW}.  
\begin{remark}
\label{P2 troubles}
On the other hand, if $M$ is $\P^2$-reducible, then we do not know whether or not $\pi_1(M)$ is circularly orderable in general.  Our techniques for producing circular orderings of $3$-manifold groups make frequent use of   algebraic properties that depend heavily on the situation at hand.  For example, our main tools use left-orderability of infinite index subgroups of $\pi_1(M)$ as in Proposition \ref{COS}, or a decomposition of $\pi_1(M)$ as a free product with amalgamation of groups whose circular orderings are well understood as in Theorem \ref{CO slope} (whose proof also uses Proposition \ref{COS}).  Neither of these properties hold for the fundamental groups of $\P^2$-reducible manifolds, in particular their fundamental groups always contain a subgroup isomorphic to $\mathbb{Z}/2 \mathbb{Z}$ \cite[Theorem 8.2]{Ep}, and so any argument that relies on left-orderability of infinite index subgroups will fail.
\end{remark}

It is well-known that left-orderability behaves in a very special way with respect to fundamental groups of irreducible $3$-manifolds.  One of the key results in the area is the following, which we present alongside a generalization to the case of circular orderability. 

\begin{proposition}\label{COS}
Let $M$ be a compact, connected, $\P^2$-irreducible $3$-manifold, let $G$ be a nontrivial group and suppose there exists an epimorphism $\r:\pi_1(M)\longrightarrow G$.
\begin{enumerate}
\item If $G$ is left-orderable then $\pi_1(M)$ is left-orderable.
 \item If $G$ is infinite and circularly orderable, then $\pi_1(M)$ is circularly orderable.
 \end{enumerate}
\end{proposition}
\begin{proof} The first claim is \cite[Theorem 3.2]{BRW}.  For the second, since $G=\r(\pi_1(M))$ is infinite then $K={\rm ker}(\r)$ is an infinite index normal subgroup of $\pi_1(M)$. From \cite[Proof of Theorem 3.2]{BRW}, the subgroup $K$ is therefore locally indicable and thus left-orderable. This means we can use the short exact sequence $1\longrightarrow K\longrightarrow \pi_1(M)\longrightarrow G\longrightarrow 1$ to construct a lexicographic circular ordering of $\pi_1(M)$ as in Proposition \ref{sesprop}(2).
\end{proof}

This implies, for instance, that $\pi_1(M)$ is left-orderable whenever $M$ satisfies the hypotheses of Theorem \ref{COS} and $H_1(M; \mathbb{Z})$ is infinite.  The case of interest is therefore when $H_1(M; \mathbb{Z})$ is finite, where the L-space conjecture posits a connection between left-orderability of $\pi_1(M)$, the existence of co-orientable taut foliations in $M$, and whether or not $M$ is a Heegaard-Floer homology L-space (that is, a manifold whose Heegaard-Floer homology is of minimal rank).

\begin{conjecture}[The L-space conjecture, \cite{BGW, Ju}]  If $M$ is an irreducible, rational homology $3$-sphere other than $\mathbb{S}^3$, then the following are equivalent:\footnote{We require the caveat that $M \neq \mathbb{S}^3$ because by our definition, the trivial group is left-orderable.}
\label{lspace}
\begin{enumerate}
\item The fundamental group of $M$ is left-orderable.
\item The manifold $M$ supports a coorientable taut foliation.
\item The manifold $M$ is not an L-space.
\end{enumerate}
\end{conjecture}

Surrounding this conjecture, there are many tools and techniques to obstruct left-orderability of fundamental groups, see \cite{CaD, DPT, CW, Ba} to name a few, and many more to prove that fundamental groups are left-orderable.  Our main contribution in this section is to connect circular orderability to this conjecture as in Theorem \ref{LO_covers} of the introduction, by applying Theorem \ref{eulerclassorder} and Proposition \ref{COS}.  Recall that a \textit{finite cyclic cover} is, by definition, a regular covering space for which the group of deck transformations is a finite cyclic group.

\begin{proof}[Proof of Theorem \ref{LO_covers}] If $H_1(M; \mathbb{Z})$ is infinite, then there exists a surjection $\pi_1(M) \rightarrow \mathbb{Z}$.  Therefore $\pi_1(M)$ is left-orderable by \cite[Theorem 1.1]{BRW}.  It follows that $\pi_1(M)$ is circularly orderable, and all finite cyclic covers (including the trivial cover) have left-orderable fundamental group, so there is nothing to prove in this case.

On the other hand, suppose $H_1(M; \mathbb{Z})$ is finite and $\pi_1(M)$ is circularly orderable with circular ordering $c$.  Note that by a standard Euler characteristic argument (see for example \cite[Lemma 3.3]{BRW}) $M$ is either closed and orientable, or $M$ has nonempty boundary containing only $\mathbb{S}^2$ and $\mathbb{P}^2$ components.  Since $M$ is $\P^2$-irreducible the latter case does not occur.

We conclude $H_1(M;\mathbb{Z}) \cong H^2(M; \mathbb{Z})$ by Poincar\'{e} duality, and then $H^2(M; \mathbb{Z}) \cong H^2(\pi_1(M); \mathbb{Z})$ since $M$ is irreducible.  Thus $[f_c]$ has finite order, say it has order $k$.  Then by Theorem \ref{eulerclassorder}, $\pi_1(M)$ admits a normal, left-orderable subgroup $H$ such that $\pi_1(M)/H$ is cyclic.  In this case, the cover $\widetilde{M}$ of $M$ with $\pi_1(\widetilde{M}) =H$ has the desired properties.

On the other hand, suppose that $p: \widetilde{M} \rightarrow M$ is a finite cyclic cover with left-orderable fundamental group.  Then there is a short exact sequence
\[ 1 \rightarrow p_*(\pi_1(\widetilde{M})) \rightarrow \pi_1(M) \rightarrow \mathbb{Z}/k\mathbb{Z} \rightarrow 1
\] 
for some $k \geq 1$, where the kernel is left-orderable and the quotient (if nontrivial) is circularly orderable.  If the quotient is trivial this means that $\pi_1(M)$ itself is left-orderable, and the conclusion follows.  If the quotient is nontrivial, then Proposition \ref{sesprop}(2) finishes the proof.
\end{proof}

Circular orderability is therefore one possible approach to the L-space conjecture, by first tackling the conjecture up to finite cyclic covers.  That is, we have the following ``circular orderability'' version of Conjecture \ref{lspace}.

\begin{conjecture}[The L-space conjecture, circular orderability version]
\label{circular_lspace}
If $M$ is an irreducible, rational homology $3$-sphere that is not a lens space, then the following are equivalent:\footnote{If we allow $M$ to be a lens space, then the conjecture as stated would not be true, again because the trivial group is left-orderable by our conventions.}
\begin{enumerate}
\item The fundamental group of $M$ is circularly orderable.
\item There exists a finite cyclic cover $\widetilde{M}$ of $M$ that supports a coorientable taut foliation.
\item There exists a finite cyclic cover $\widetilde{M}$ of $M$ that is not an L-space.
\end{enumerate}
\end{conjecture}

Note that this has conjectural implications beyond what would follow from the ``left-orderability'' version of the conjecture.  For instance, the following theorem connects certain topological properties directly to circular orderability (without passing via left-orderability).

\begin{theorem}\cite[Theorem 3.1, Corollary 3.9]{CaD}
If $M$ is an orientable, atoriodal $3$-manifold containing a very full tight essential lamination or supporting a pseudo-Anosov flow, then $\pi_1(M)$ is circularly orderable.  
\end{theorem}

Conjecture \ref{circular_lspace} therefore predicts that one can obstruct the existence of a pseudo-Anosov flow on a rational homology $3$-sphere $M$, or the existence of a very full tight essential lamination in $M$ by showing, for example, that all finite cyclic covers are L-spaces or that no finite cyclic cover supports a coorientable taut foliation.  

\section{Slope detection and $3$-manifolds having nontrivial JSJ decomposition}\label{sec1}

One of the key techniques in the analysis of manifolds admitting a nontrivial JSJ decomposition is that of \textit{slope detection} by left-orderings.  This idea first appeared in a basic form in \cite[Definition 2.6]{CLW}, was further developed in \cite{BC} as a key technique in proving Conjecture \ref{lspace} for graph manifolds \cite{BC, HRRW}, and appears also in \cite{BC2, BGH}.  We recall the central technique from \cite{CLW} in the following theorem.

\begin{theorem} \cite[Theorem 2.7]{CLW} 
\label{lo_slopes}
Suppose that $M_1, M_2$ are $3$-manifolds with incompressible torus boundaries $\partial M_i$, and that $\phi : \partial M_1 \rightarrow \partial M_2$ is a homeomorphism such that $W = M_1 \cup_{\phi} M_2$ is irreducible.  If there exists a slope $\alpha$ such that both $\pi_1(M_1(\alpha))$ and $\pi_1(M_2(\phi_*(\alpha)))$ are left-orderable, then $\pi_1(W)$ is left-orderable.
\end{theorem}

In what follows, we develop a generalization of this technique that applies to circular orderability and present applications to various classes of manifolds. To prepare, we recall the following construction.  

If $(G,c)$ and $(H,d)$ are circularly ordered groups, and $\phi : G \rightarrow H$ is a homomorphism satisfying 
\[ c(g_1, g_2, g_3) = d(\phi(g_1), \phi(g_2), \phi(g_3)) \mbox{ for all $g_1, g_2, g_3 \in G$}
\]
then we say that $\phi$ is \textit{order-preserving} or \textit{compatible with the pair $(c,d)$}.  Note that in this case, $\phi$ is necessarily injective.  Then we can define $\widetilde{\phi} : \widetilde{G_c} \rightarrow \widetilde{H_d}$ by $\widetilde{\phi}(n, g) = (n, \phi(g))$, so that $\widetilde{\phi}$ is order-preserving with respect to the left-orderings $<_c$ and $<_d$ of $\widetilde{G_c} $ and  $\widetilde{H_d}$ respectively (or \textit{compatible with the pair $(<_c, <_d)$}).

\begin{proposition}
\label{amalgam_prop}
Suppose that $\{(G_i, c_i)\}_{i \in I}$ are circularly ordered groups each containing a subgroup $H_i$.  If $(D, d)$ is a cyclic circularly ordered group and $\phi_i : D \rightarrow H_i$ is an isomorphism compatible with the pair $(d, c_i)$ for every $i$, then the free product with amalgamation $*_{i \in I} G_i(D \stackrel{\phi_i}{\cong} H_i) $ is circularly orderable.

\end{proposition}
\begin{proof}

The homomorphism $\phi_i \phi_j^{-1}$ is compatible with the pair $(c_j, c_i)$ for all $i, j \in I$, and therefore the map $\widetilde{\phi_i \phi_j^{-1}}$ is also compatible with the pair $(<_{c_j}, <_{c_i})$ for all $i, j \in I$.  

Now suppose that $D \cong \mathbb{Z}$, in which case $\widetilde{D_d} \cong \mathbb{Z} \times \mathbb{Z}$ and the images $\widetilde{\phi_i}(\widetilde{D_d})$ and $\widetilde{\phi_j}(\widetilde{D_d})$ are neither bounded from above, nor from below, in $\widetilde{(G_i)_{c_i}}$ and $\widetilde{(G_j)_{c_j}}$ respectively, since each image contains the cofinal central element of the respective extension. We may therefore apply \cite[Proposition 5.6]{CG} to conclude that the group $*_{i \in I} \widetilde{(G_i)_{c_i}}(\widetilde{D_d} \stackrel{\widetilde{\phi_i}}{\cong} \widetilde{(H_i)_{c_i}}) $ is left-orderable, and then apply \cite[Theorem 1]{CG} to conclude that $*_{i \in I} G_i(D \stackrel{\phi_i}{\cong} H_i) $ is circularly orderable in this case.

Next suppose that $D$ is finite, in which case $\widetilde{D_d}$ is infinite cyclic.  Then  $*_{i \in I} \widetilde{(G_i)_{c_i}}(\widetilde{D_d} \stackrel{\widetilde{\phi_i}}{\cong} \widetilde{(H_i)_{c_i}}) $ is an amalgamation of left-orderable groups along a cyclic subgroup, which is always left-orderable \cite{BG}.  As above, that $*_{i \in I} G_i(D \stackrel{\phi_i}{\cong} H_i) $ is circularly orderable then follows from \cite[Theorem 1]{CG}.
\end{proof}

In preparation for the next theorem, suppose that $G$ is a group and let $g \in G$.  We use the notation $\langle \langle g \rangle \rangle$ to denote the normal closure of $g \in G$, that is, the smallest (with respect to inclusion) normal subgroup of $G$ containing $g$.

\begin{theorem}\label{CO slope}
Let $M_1$, $M_2$ be two $3$-manifolds with incompressible torus boundaries, and let $\phi:\partial M_1\rightarrow \partial M_2$ be a homeomorphism such that $M=M_1\cup_{\phi} M_2$ is $\mathbb{P}^2$-irreducible.  If there exists a slope $\a \in H_1(\partial M_1; \mathbb{Z})/\{ \pm 1\}$ such that $\pi_1(M_1(\a))$ and $\pi_1(M_2(\phi_*(\a)))$ are infinite circularly orderable groups, and either
\begin{enumerate}
\item at least one of $\pi_1(\partial M_1) \subset \langle \langle \a \rangle \rangle $ or $\pi_1(\partial M_2) \subset \langle \langle \phi_*(\a) \rangle \rangle $ holds, or
 \item  $\pi_1(M_1(\a))$ and $\pi_1(M_2(\phi_*(\a)))$ admit circular orderings $c_1$ and $c_2$ respectively such that the induced map
\[
\overline{\phi_*} : \pi_1(\partial M_1) /(\langle \langle \a \rangle \rangle \cap \pi_1(\partial M_1))  \longrightarrow \pi_1(\partial M_2)/(\langle \langle \phi_*(\a) \rangle \rangle \cap \pi_1(\partial M_2))
\]
is an isomorphism between nontrivial groups which is compatible with the pair $(c_1, c_2)$, 
\end{enumerate}
then $\pi_1(M)$ is circularly orderable.

\end{theorem}
\begin{proof} 
Suppose first that there exists a slope $\alpha$ satisfying condition (1) of the theorem, without loss of generality we assume that $\pi_1(\partial M_1) \subset \langle \langle \a \rangle \rangle $. To simplify notation, let $G_i=\pi_1(M_i)$, $i=1,2$, each equipped with an inclusion homomorphism $f_i:\Z\oplus \Z\longrightarrow G_i$ that identifies the peripheral subgroup $\pi_1(\partial M_i)$ with $\Z\oplus \Z$, satisfying $\phi_\ast\circ f_1=f_2$, and let $q_1:G_1\longrightarrow G_1/\langle \langle \a \rangle \rangle$ and $q_2:G_2\longrightarrow G_2/\langle \langle \phi_\ast (\a) \rangle \rangle$ be the quotient maps.

In this case, there exists a unique map $f$ such that the following diagram commutes:
\[
\xymatrix@R=1.2cm@C=1.2cm{
 \Z\oplus \Z \ar[r]^{f_2}      \ar[d]_{f_1}         & G_2 \ar[d]_{} \ar@/^1pc/[rdd]^{1}   \\
 G_1 \ar[r]^{} \ar@/_1pc/[rrd]_{q_1} & G_1\ast_{\phi_\ast}G_2 \ar@{-->}[rd]|{f}             \\
                                      &                                   & G_1/\langle \langle \a \rangle \rangle
}
\]
As the image of the map $f$ is an infinite, circularly orderable group, that $\pi_1(M) \cong G_1\ast_{\phi_\ast}G_2$ is circularly orderable follows from Proposition \ref{COS}(2).

Suppose there exists a slope $\alpha$ satisfying the conditions (2) of the theorem.  Let $\overline{\phi_*} $ denote the map induced by $\phi$, as in the statement of the theorem.  There are two possibilities for subgroup $\langle \langle \a \rangle \rangle \cap \pi_1(\partial M_1)$ of $\pi_1(\partial M_1) \cong \mathbb{Z} \oplus \mathbb{Z}$, it is isomorphic to either $\mathbb{Z}$ or $\mathbb{Z} \oplus n \mathbb{Z}$ for some $n >1$.   In each of these cases, $\overline{\phi_*} $ is an isomorphism between cyclic subgroups of $G_1$ and $G_2$ that is compatible with $(c_1, c_2)$ and so $G_1/\langle \langle \a \rangle \rangle *_{\overline{\phi_*}} G_2/\langle \langle \phi_*(\a) \rangle \rangle$ is circularly orderable by Proposition \ref{amalgam_prop}.

Then as before, there is a unique map $f$ such that the following diagram commutes:

\[
\xymatrix@R=1.2cm@C=1.2cm{
 \Z\oplus \Z \ar[r]^{f_2}      \ar[d]_{f_1}         & G_2 \ar[d]_{} \ar@/^1pc/[rdd]^{q_2}   \\
 G_1 \ar[r]^{} \ar@/_1pc/[rrd]_{q_1} & G_1\ast_{\phi_\ast}G_2 \ar@{-->}[rd]|{f}             \\
                                      &                                   & G_1/\langle \langle \a \rangle \rangle *_{\overline{\phi_*}} G_2/\langle \langle \phi_*(\a) \rangle \rangle
}
\]

Now since the groups $\pi_1(M_1(\a))$ and $\pi_1(M_2(\phi_*(\a)))$ are infinite, so is $G_1/\langle \langle \a \rangle \rangle *_{\overline{\phi_*}} G_2/\langle \langle \phi_*(\a) \rangle \rangle$.  That $\pi_1(M) \cong G_1\ast_{\phi_\ast}G_2$ is circularly orderable follows from Proposition \ref{COS}.

\end{proof}

Note that this recovers Theorem \ref{lo_slopes} in the case that the quotients are both left-orderable.

Recall that $\mathrm{rot}_c : G \rightarrow \mathbb{S}^1$ is an order-preserving homomorphism upon restriction to any abelian subgroup, so long as it is injective.  From this observation and Theorem \ref{CO slope} we arrive at the following:

\begin{corollary}
\label{match rot numbers}
Let $M_1$, $M_2$, $\phi:\partial M_1\rightarrow \partial M_2$ and $M=M_1\cup_{\phi} M_2$ be as in Theorem \ref{CO slope}, and that $\a \in H_1(\partial M_1; \mathbb{Z})/\{ \pm 1\}$ is such that $\pi_1(M_1(\a))$ and $\pi_1(M_2(\phi_*(\a)))$ are infinite circularly orderable groups.   Let $\beta \in \pi_1(\partial M_1)$ denote a dual class to $\alpha$, and let $q_1 : \pi_1(M_1) \rightarrow \pi_1(M_1(\a))$ and $q_2: \pi_1(M_2) \rightarrow \pi_1(M_2(\phi_*(\a)))$ denote the quotient maps.  If there exist circular orderings $c_1$ and $c_2$ of $\pi_1(M_1(\a))$ and $\pi_1(M_2(\phi_*(\a)))$ respectively such that 
\[ c_1(q_1(\beta^j), q_1(\beta^k), q_1(\beta^\ell)) = c_2(q_2(\phi_*(\beta^j)), q_2(\phi_*(\beta^k)), q_2(\phi_*\beta^\ell))) 
\]
for all $j, k , \ell \in \mathbb{Z}$, then $\pi_1(M)$ is circularly orderable.  In particular, if 
\[ \mathrm{rot}_{c_1}(q_1(\beta)) = \mathrm{rot}_{c_2}(q_2(\phi_*(\beta)))
\]
and $\mathrm{rot}_{c_i}$ are injective, then $\pi_1(M)$ is circularly orderable.
\end{corollary}

Applications of Theorem \ref{CO slope} or Corollary \ref{match rot numbers} therefore hinge upon being able to construct circular orderings of fundamental groups of $3$-manifolds where certain elements have prescribed rotation number.

\section{Rational longitudes and knot manifolds}
\label{rat_long}

Recall that a knot manifold is a compact, connected, irreducible and orientable $3$-manifold with boundary an incompressible torus.  In this section we demonstrate a technique for creating circular orderings of fundamental groups of knot manifolds, where the cyclic subgroup generated by a class dual to the rational longitude has a prescribed circular ordering.  

\begin{lemma}
\label{extend}
Suppose that $C$ is a cyclic group, $D$ is a nontrivial subgroup of $C$ and $c$ is a circular ordering of $D$.  Then there exists a circular ordering $c'$ of $C$ such that 
\[ c'(r, s, t)= c(r,s,t) \mbox{ for all $r, s, t \in D$}.
\]
\end{lemma}
\begin{proof} 
We first consider the case where $C = \mathbb{Z}$ is infinite cyclic and $D = k \mathbb{Z}$.  Let $c$ be an arbitrary circular ordering of $k \mathbb{Z}$, and let $d$ denote the standard circular ordering of $\mathbb{S}^1$.  Consider the rotation number homomorphism $\mathrm{rot}_c:k \mathbb{Z} \rightarrow \mathbb{S}^1$ corresponding to the circular ordering $c$.  Suppose that $\mathrm{rot}_c(k) = exp(2 \pi i \alpha)$ where $\alpha \in [0,1]$, and define $\phi: \mathbb{Z} \rightarrow \mathbb{S}^1$ by $\phi(1) = exp(\frac{2 \pi i \alpha}{k})$.  There are two cases to consider.

First, if $\mathrm{rot}_c$ is injective, then $c(r,s,t) = d(\mathrm{rot}_c(r), \mathrm{rot}_c(s), \mathrm{rot}_c(t))$ for all $r, s, t \in k \mathbb{Z}$.  Note that $\phi$ is injective because $\mathrm{rot}_c$ is injective, so we can define a circular ordering $c'$ on $\mathbb{Z}$ by $c'(r,s,t) = d(\phi(r), \phi(s), \phi(t))$, which clearly extends $c$ as required.

On the other hand suppose that $\mathrm{rot}_c$ is not injective, say $\alpha= p/q$ with $\gcd(p,q)=1$.  Then $c$ arises lexicographically from a short exact sequence
\[ 1 \longrightarrow H \longrightarrow  k \mathbb{Z} \stackrel{\mathrm{rot_c}}{\longrightarrow} \mathbb{Z}/q\mathbb{Z} \longrightarrow 1, 
\]
where $\mathbb{Z}/q\mathbb{Z}$ is identified naturally with the $q^{th}$ roots of unity and equipped with the restriction of the natural circular ordering $d$ of $\mathbb{S}^1$.  In this case $\phi$ yields a short exact sequence
\[1 \longrightarrow H \longrightarrow   \mathbb{Z} \stackrel{\phi}{\longrightarrow} \mathbb{Z}/ qk \mathbb{Z} \longrightarrow 1
\]
where $\mathbb{Z}/ qk \mathbb{Z}$ is again equipped with the natural circular ordering arising from the natural embedding into $\mathbb{S}^1$.  Thus if we use the latter short exact sequence to lexicographically define a circular ordering $c'$ of $\mathbb{Z}$, using the same left-ordering of $H$ as in the former short exact sequence, then $c'$ will be an extension of the given circular ordering $c$ of $k \mathbb{Z}$.

When $C$ is finite, $\mathrm{rot}_c$ is injective, so we can use the same construction as in the first case above.
\end{proof}

The next lemma is a standard result, but it is essential to our arguments and so we include a proof.

\begin{lemma}\label{rank}
 Let $M$ be a compact, connected, orientable 3-manifold with a torus boundary, and $\mathbb{F}$ be a field. If $i: \partial M \rightarrow M$ denotes the inclusion map, then the image of the map $$i_*^1: H_1(\partial M; \mathbb{F})\longrightarrow H_1(M; \mathbb{F})$$ is of rank one.
\end{lemma}
\begin{proof}
Considering the pair $(M, \partial M)$ we have the following long exact sequence,


 
 $\cdots \longrightarrow H_1(\partial M; \mathbb{F})\stackrel{i_*^1}{\longrightarrow} H_1(M; \mathbb{F})\stackrel{p_*^1}{\longrightarrow} H_1(M, \partial M; \mathbb{F})\stackrel{\partial_1}{\longrightarrow} H_0(\partial M; \mathbb{F})\stackrel{i_*^0}{\longrightarrow} H_0(M; \mathbb{F})\stackrel{\partial_0}{\longrightarrow} 0.$

 Since $M$ is connected, it is also path connected, and hence $ H_0(M, \partial M; \mathbb{F})=0$. By exactness, we have that $im(i_*^1)=ker(p^1_*)$, $im(p_*^1)=ker(\partial_1)$, $im(\partial_1)=ker(i_*^0)$ and $im(i_*^0)=H_0(M; \mathbb{F})$. By the Rank-Nullity Theorem,
 $$dim_{\mathbb{F}}(im(i_*^1)) =dim_{\mathbb{F}}(H_1(M; \mathbb{F})) -dim_{\mathbb{F}}(H_1(M, \partial M; \mathbb{F}))+dim_{\mathbb{F}}(H_0(\partial M; \mathbb{F}))-dim_{\mathbb{F}}(H_0(M; \mathbb{F})).$$
 Since $M$ is compact, by Universal Coefficient Theorems and duality we have that $$dim_{\mathbb{F}}(H_1(M, \partial M; \mathbb{F}))=dim_{\mathbb{F}}(H_2(M; \mathbb{F})).$$ Hence,
 \begin{align*}
 dim_{\mathbb{F}}(im(i_*^1)) &=dim_{\mathbb{F}}(H_1(M; \mathbb{F})) - dim_{\mathbb{F}}(H_1(M, \partial M; \mathbb{F})) + dim_{\mathbb{F}}(H_0(\partial M; \mathbb{F})) - dim_{\mathbb{F}}(H_0(M; \mathbb{F}))\\
 &= -\chi(M) + dim_{\mathbb{F}}(H_3(M; \mathbb{F})) + dim_{\mathbb{F}}(H_0(\partial M; \mathbb{F})).
 \end{align*}
 Since the boundary of $M$ is not empty, $dim_{\mathbb{F}}(H_3(M; \mathbb{F}))=0$. Thus,
 $$dim_{\mathbb{F}}(im(i_*^1))=-\chi(M) + dim_{\mathbb{F}}(H_0(\partial M; \mathbb{F})).$$
 
From this, we note that $\chi(M) = \frac{1}{2}\chi(\partial M) =0$ and so we conclude that $dim_{\mathbb{F}}(im(i_*^1)) = dim_{\mathbb{F}}(H_0(\partial M; \mathbb{F})) = 1$.
%
 
\end{proof}

It follows that the image of $i_*^1: H_1(\partial M; \mathbb{Z})\longrightarrow H_1(M; \mathbb{Z})$ is also rank one when $M$ is orientable.  The unique primitive element in $H_1(\partial M; \mathbb{Z})$ whose image is of finite order in $H_1(M; \mathbb{Z})$ is referred to as the \textit{rational longitude} of $M$, and is denoted $\lambda_M$. 

\begin{remark}
Consider the following long exact sequence with $M$ as in the previous lemma.

$\cdots \longrightarrow H_1(\partial M; \mathbb{Z})\stackrel{i_*^1}{\longrightarrow} H_1(M; \mathbb{Z})\stackrel{p_*^1}{\longrightarrow} H_1(M, \partial M; \mathbb{Z})\stackrel{\partial_1}{\longrightarrow} H_0(\partial M; \mathbb{Z})\stackrel{i_*^0}{\longrightarrow} H_0(M; \mathbb{Z})\stackrel{\partial_0}{\longrightarrow}\\
 H_0(M, \partial M; \mathbb{Z})\longrightarrow 0.$
 
Since $M$ and $\partial M$ are both connected, they are also path connected, and hence $ H_0(M, \partial M; \mathbb{Z})=0$, $H_0(\partial M; \mathbb{Z})\cong \mathbb{Z}$ and $H_0(M; \mathbb{Z})=\mathbb{Z}$. Hence $i_*^0$ is an isomorphism, and $im(\partial_1)=ker(i_*^0)=0$. This implies that $im(p_*^1)=ker(\partial_1)=H_1(M, \partial M; \mathbb{Z})$, and that $i_*^1$ is surjective if and only if $H_1(M, \partial M; \mathbb{Z})=0$.  Thus we do not assume surjectivity of $i_*^1$, which necessitates the use of Lemma \ref{extend} in the proofs below.
\end{remark}

\begin{remark}
    If $M$ is a compact, non-orientable, $\mathbb{P}^2$-irreducible $3$-manifold, then $\pi_1(M)$ is left-orderable by \cite[Lemma 3.3 and Theorem 3.1]{BRW}; on the other hand if $M$ is $\mathbb{P}^2$-reducible then our techniques do not apply in general, see Remark \ref{P2 troubles}. Hence, from now on, we assume that all $3$-manifolds are orientable.
\end{remark}

\begin{proposition}
\label{rat_rot_num}
Suppose that $M$ is a knot manifold, let $\mu$ be the class of any closed curve in $\partial M$ that is dual to $\lambda_M$ and let $q$ denote the quotient map $q: \pi_1(M) \rightarrow \pi_1(M(\lambda_M))$.  If $M(\lambda_M)$ is irreducible, then for every circular ordering $c'$ of the cyclic subgroup $\langle q(\mu) \rangle$, there exists a circular ordering $c$ of $\pi_1(M(\lambda_M))$ such that 
\[ c'(q(\mu^j), q(\mu^k), q(\mu^\ell)) = c(q(\mu^j), q(\mu^k), q(\mu^\ell)) 
\]
for all $j, k, \ell \in \mathbb{Z}$.
\end{proposition}
\begin{proof}
For such a manifold $M$, it follows that $|H_1(M(\lambda_M); \mathbb{Z})|$ is infinite, with the class of $\mu$ being infinite order by Lemma \ref{rank}.  Therefore there exists a map $\psi: \pi_1(M(\lambda_M)) \rightarrow \mathbb{Z}$ with the image of $\mu$ being nontrivial, say $\psi( \mu ) = k$.  Let $c'$ be a given circular ordering of $\langle q(\mu) \rangle$.

Denote the the standard circular ordering of $ \mathbb{S}^1$ by $d$, and suppose that $\mathrm{rot}_{c'}(q(\mu)) = r$.  Define a map $\phi : k \mathbb{Z} \rightarrow  \mathbb{S}^1$ by $\phi(k) = \exp(2 \pi i r)$, and first suppose that $\mathrm{rot}_{c'}$ is injective.  Then $\phi $ is injective, so we may define a circular ordering $c''$ of $k \mathbb{Z}$ by $c''(r, s, t) = d(\phi(r), \phi(s), \phi(t))$.  Next suppose that $\mathrm{rot}_{c'}$ is not injective, so that $r = \frac{s}{t} \in \mathbb{Q}$ (with $\frac{s}{t}$ in lowest terms) and the circular ordering $c'$ of $\langle q(\mu) \rangle$ is lexicographic with respect to the short exact sequence
\[ 0 \rightarrow \ker(\mathrm{rot}_{c'} ) \rightarrow \langle q(\mu) \rangle \rightarrow \mathbb{Z}/t\mathbb{Z} \rightarrow 0
\]
for some choice of left-ordering of $\ker(\mathrm{rot}_{c'})$.  Then $\phi$ is not injective, and we may use Proposition \ref{sesprop}(2) to create a circular ordering $c''$ of $k \mathbb{Z}$ from the sequence $1 \rightarrow K \rightarrow k\mathbb{Z} \stackrel{\phi}{\rightarrow}  \mathbb{S}^1$, where $K$ is the kernel of $\phi$, such that our choice of left-ordering of $K$ agrees with the left-ordering of $\ker(\mathrm{rot}_{c'} )$ under the isomorphism $k\mathbb{Z} \cong \langle q(\mu) \rangle$ given by $k \mapsto q(\mu)$. 

In either case, by Lemma \ref{extend}, there is a circular ordering $\hat{c}$ of $\psi(\pi_1(M(\lambda_M))) \cong \mathbb{Z}$ that extends $c''$.  That $\pi_1(M(\lambda_M))$ admits a circular ordering $c$ with the required property now follows from Proposition \ref{sesprop}(2) and the proof of Proposition \ref{COS}(2).
\end{proof}
 
Following \cite[p.428--431]{Scot1}, recall that a Seifert fibred space is a $3$-manifold which is foliated by circles, called fibres, such that each circle $S$ has a closed tubular neighborhood which is a union of fibres and is isomorphic to a fibred solid torus if $S$ preserves the orientation or a fibred solid Klein bottle if $S$ reverses the orientation.  We use $h$ throughout to denote the class of the regular fibre.

\begin{proposition}
\label{rat_long_CO}
Suppose that $M_1$ and $M_2$ are compact, connected and orientable $3$-manifolds with torus boundary, that $\phi: \partial M_1 \rightarrow \partial M_2$ is a homeomorphism, and $M = M_1 \cup_{\phi} M_2$ is irreducible.
\begin{enumerate} 
\item If $\mathrm{rk}(H_1(M_1; \mathbb{Q})) \geq 2$ then $\pi_1(M)$ is left-orderable. 
\item  Suppose $\mathrm{rk}(H_1(M_1; \mathbb{Q})) =1$ and let $\lambda_1$ denote the rational longitude of $M_1$. If \\ $\pi_1(M_2(\phi_*(\lambda_1)))$ is infinite and circularly orderable and either $M_1(\lambda_1)$ is irreducible, or $M_1$ is Seifert fibred with incompressible boundary, then  $\pi_1(M)$ is circularly orderable.
\end{enumerate}
\end{proposition}
\begin{proof} In what follows, let $i :\partial M_1 \rightarrow M_1$ denote the inclusion map. 

 To prove (1), assume $\mathrm{rk}(H_1(M_1; \mathbb{Q})) \geq 2$.  In this case, the image of $H_1(\partial M_1; \mathbb{Z}) \stackrel{i_*}{\rightarrow} H_1(M_1; \mathbb{Z})$ is rank one by Lemma \ref{rank}.  Therefore we may compose the Hurewicz map $\pi_1(M_1) \rightarrow H_1(M; \mathbb{Z})$ with the quotient $H_1(M; \mathbb{Z}) \rightarrow H_1(M; \mathbb{Z})/im(i_*)$ and obtain a map $\psi: \pi_1(M_1) \rightarrow A$, where $A$ is an abelian group of positive rank, satisfying $\psi( i_*(\pi_1( \partial M_1))) = 0$.  In this case, since $M$ is a free product with amalgamation of $\pi_1(M_1)$ and $\pi_1(M_2)$, there exists a surjective homomorphism $\pi_1(M) \rightarrow A$ induced by $\psi : \pi_1(M_1) \rightarrow A$ and the zero homomorphism $\pi_1(M_2) \rightarrow A$.  Thus $\pi_1(M)$ admits a torsion-free abelian quotient, so $\pi_1(M)$ is left-orderable by \cite[Theorem 3.2]{BRW}.

On the other hand, suppose that $\mathrm{rk}(H_1(M_1; \mathbb{Q})) =1$ and let $q_1 : \pi_1(M_1) \rightarrow \pi_1(M_1(\lambda_1))$ and $q_2: \pi_1(M_2) \rightarrow \pi_1(M_2(\phi_*(\lambda_1)))$ denote the quotient maps.  Equip $\pi_1(M_2(\phi_*(\lambda_1)))$ with a circular ordering $c_2$, let $\mu_2$ denote the class of a curve dual to $\phi(\lambda_1)$.  Let $\mu_1$ denote the class of a curve dual to $\lambda_1$.  In this case, if $M_1(\lambda_{1})$ is irreducible then by Proposition \ref{rat_rot_num} we can equip $\pi_1(M_1(\lambda_1))$ with a circular ordering $c_1$ that agrees with $c_2$ under the identification of $\langle q_1(\mu_1) 
\rangle $ and $\langle q_2(\mu_2) \rangle$ induced by $\phi_*$. The result now follows from Corollary \ref{match rot numbers}.


On the other hand, if $M_1$ is Seifert fibred then $M_1(\lambda_1)$ is irreducible by \cite{Hei2} unless $\lambda_1$ is the class of a regular fibre; so we need only consider the case that $\lambda_1$ is the class of a regular fibre.   For $\lambda_1$ to be the class of a regular fibre, $M_1$ must have nonorientable base orbifold.  In this case, since $\mathrm{rk}(H_1(M_1; \mathbb{Q})) =1$ we know the base orbifold must be a once-punctured projective plane and we compute 
\[\pi_1(M_1) = \langle a, \gamma_1, \dots, \gamma_n, \mu, h \mid a h a^{-1} =h^{-1}, [\gamma_i, h] = 1,  [\mu, h] = 1, \gamma_i^{\alpha_i} = h^{\beta_i}, a^2 \mu \gamma_1 \dots \gamma_n = 1 \rangle
\]
where $\mu$ is a class dual to $\lambda_1 =h$ on $\partial M_1$. Therefore
\[ \pi_1(M_1(\lambda_1)) = \langle a, \gamma_1, \dots, \gamma_n, \mu \mid  \gamma_i^{\alpha_i} = 1, a^2 \mu \gamma_1 \dots \gamma_n = 1 \rangle
\] 
which is isomorphic to the free product with amalgamation
\[  \langle \gamma_1 \mid \gamma_1^{\alpha_1} = 1\rangle * \dots * \langle \gamma_n \mid \gamma_n^{\alpha_n} = 1\rangle * \langle \mu \rangle*_{ \mu \gamma_1 \dots \gamma_n  = a^{-2}} \langle a \rangle.
\]
In this case, we can first equip the infinite cyclic group $ \langle \mu \rangle$ with a circular ordering $c$ that agrees with $c_2$ under the identification of $\langle \mu 
\rangle $ and $\langle q_2(\mu_2) \rangle$ induced by $\phi_*$.  By Proposition \ref{free product} this circular ordering extends to the free product $\langle \gamma_1 \mid \gamma_1^{\alpha_1} = 1\rangle * \dots * \langle \gamma_n \mid \gamma_n^{\alpha_n} = 1\rangle * \langle \mu \rangle$, which in turn extends to a circular ordering of $\pi_1(M_1(\lambda_1))$ by combining Lemma \ref{extend} and Proposition \ref{amalgam_prop}.  Thus the claim follows from Corollary \ref{match rot numbers}.
 \end{proof}

\section{Seifert fibred manifolds and graph manifolds}
\label{seifert and graph}

Aside from this last tool, there are other situations where we can control the rotation numbers of certain elements in circular orderings of fundamental groups.  We next investigate Seifert fibred manifolds, where our goal is to show that the fundamental group of a Seifert fibred space is always circularly orderable whenever it is infinite, and to describe the possible circular orderings in terms of the rotation numbers of the class of a regular fibre.  

\subsection{Seifert fibred manifolds}
It was first claimed by Calegari in \cite[Remark 4.3.2]{Cal} that if $M$ is Seifert fibred and $\pi_1(M)$ is infinite, then it is circularly orderable.  We provide the details of this claim below.

Recall that when $M$ is an orientable Seifert fibred manifold over an orientable closed surface of genus $g \geq 0$, the fundamental group has presentation
\[ \pi_1(M) = \langle a_1, b_1, \dots, a_g, b_g, \gamma_1, \dots, \gamma_n, h \mid \mbox{$h$ central}, \gamma_i^{\alpha_i} = h^{\beta_i}, [a_1, b_1]\dots [a_g, b_g]\gamma_1 \dots \gamma_n = h^b \rangle
\]
and if the surface is nonorientable, then 
\[ \pi_1(M) = \langle a_1, \dots, a_{g}, \gamma_1, \dots, \gamma_n, h \mid a_j h a_j^{-1} =h^{-1}, [\gamma_i, h] = 1, \gamma_i^{\alpha_i} = h^{\beta_i}, a_1^2\dots a_{g}^2\gamma_1 \dots \gamma_n = h^b \rangle.
\]
In either case, quotienting by the normal subgroup $\langle h \rangle$ (the cyclic subgroup generated by the class of the regular fibre) yields the orbifold fundamental group of the underlying orbifold.  Consequently we observe the following lemma.

\begin{lemma}
\label{orbifold_co}
Suppose that $\alpha_1, \dots, \alpha_n \geq 2$ are integers and that $\mathcal{B}$ is an orbifold of type $\mathbb{S}^2(\alpha_1, \dots, \alpha_n)$ or  $\mathbb{P}^2(\alpha_1, \dots, \alpha_n)$. Assume that:
\begin{enumerate}
    \item $n\geq 3$ and $\Sigma_{i=1}^n \frac{1}{\a_i} <n-2$ if $\mathcal{B}=\mathbb{S}^2(\alpha_1, \dots, \alpha_n)$, and
    \item $n\geq 2$ if $\mathcal{B}=\mathbb{P}^2(\alpha_1, \dots, \alpha_n)$.
\end{enumerate}

Then $\pi_1^{orb}(\mathcal{B})$ is circularly orderable.
\end{lemma}
\begin{proof}

{\bf Case 1:} Assume that $\mathcal{B} = \mathbb{S}^2(\alpha_1, \dots, \alpha_n)$, $n\geq 3$ and $\Sigma_{i=1}^n \frac{1}{\a_i} <n-2$. Then as the orbifold $\mathcal{B}$ is hyperbolic, the group $ \pi_1^{orb}(\mathcal{B}) $ embeds in $\PSL_2(\mathbb{R})$.  As $\PSL_2(\mathbb{R})$ embeds in $\mathrm{Homeo}_+(\mathbb{S}^1)$, it is circularly orderable, so the result follows.

{\bf Case 2}: Assume that $\mathcal{B}= \mathbb{P}^2(\alpha_1, \dots, \alpha_n)$ and $n\geq 2$. For the group $\mathbb{Z}/\alpha_1 \mathbb{Z} * \dots * \mathbb{Z}/\alpha_n \mathbb{Z}$, suppose that the generator of $\mathbb{Z} /\alpha_i \mathbb{Z}$ is $x_i$, and use $x$ to denote the product $x_1 \dots x_n$.  Then
\[ \pi_1^{orb}(\mathcal{B}) = (\mathbb{Z}/\alpha_1 \mathbb{Z} * \dots * \mathbb{Z}/\alpha_n \mathbb{Z})*_{\langle x \rangle = 2 \mathbb{Z}} \mathbb{Z},
\]
and the group $\mathbb{Z}/\alpha_1 \mathbb{Z} * \dots * \mathbb{Z}/\alpha_n \mathbb{Z}$ is circularly orderable by Proposition \ref{free product}.  That $\pi_1^{orb}(\mathcal{B})$ is circularly orderable then follows from Proposition \ref{amalgam_prop}.
\end{proof}

Last, we note the following holds for all left-ordered groups admitting a cofinal, central element.

\begin{proposition}
\label{cofinal_central_gives_every_r}
Suppose that $(G,<)$ is a left-ordered group, and that $z \in G$ is a cofinal, central element.  Then for every $p \in \mathbb{N}_{>0}$, the group $G$ admits a circular ordering $c$ such that $\mathrm{rot}_c(z) = 1/p$.
\end{proposition}
\begin{proof}
Let $p \in  \mathbb{N}_{>0}$ be given, and first construct a circular ordering $d$ of $G/\langle z^p \rangle$ by mimicking Construction \ref{quotient_construction} as follows.   Define the minimal representative $\overline{g}$ of each $g\langle z^p \rangle \in G/\langle z^p \rangle$ to be the unique coset representative satisfying $id \leq \overline{g} < z^{p}$, and set
\[ d(g_1\langle z^p \rangle, g_2\langle z^p \rangle, g_3\langle z^p \rangle) = \mathrm{sign}(\sigma), 
\]
where $\sigma$ is the unique permutation such that the minimal representatives satisfy $\overline{g_{\sigma(1)}} < \overline{g_{\sigma(2)}} < \overline{g_{\sigma(3)}}$. Then construct a circular ordering $c$ of $G$ by using the short exact sequence 
\[ 1 \rightarrow \langle z^p \rangle  \rightarrow G \rightarrow G/\langle z^p \rangle \rightarrow 1.
\]
the orderings $<$ and $d$ of the kernel and quotient respectively, and Lemma \ref{sesprop}.  The circular ordering $c$ satisfies $\mathrm{rot}_c(z) = 1/p$ (c.f  \cite[Proof of Theorem 6.2]{BaC}). 
\end{proof}

\begin{proposition}
\label{cofinal_ordering}
Suppose that $M$ is a compact, connected orientable Seifert fibred space with base orbifold $\mathbb{S}^2(\alpha_1, \dots, \alpha_n)$ where $n\geq 3$, and let $h$ denote the class of a regular fibre in $\pi_1(M)$.  If $\pi_1(M)$ is left-orderable, then $\pi_1(M)$ admits a left-ordering relative to which $h$ is cofinal.
\end{proposition}
\begin{proof} This result is essentially a restatement of \cite[Proposition 4.7]{BC}.  
\end{proof}

\begin{proof}[Proof of Theorem \ref{SFCO}] We first prove (1).  By \cite[Theorem 1.3]{BRW}, $\pi_1(M)$ is left-orderable whenever the first Betti number $b_1(M)$ is positive and $M \not \cong \mathbb{P}^2 \times \mathbb{S}^1$.  When $M= \mathbb{P}^2 \times \mathbb{S}^1$, the fundamental group $\pi_1(M)$ clearly admits a circular ordering of the required kind, via a lexicographic construction.  This proves the claim for Seifert fibred manifolds $M$ with $b_1(M)>0$.

Thus we can assume that $b_1(M)=0$, which implies that $M$ is closed and orientable (See, e.g. \cite[Lemma 3.3]{BRW}), and the base orbifold of $M$ is either $\mathbb{S}^2(\a_1, \a_2, \dots, \a_n)$ or $\mathbb{P}^2(\a_1, \a_2, \dots, \a_n)$.  

Assume that either condition (1) or (2) of Lemma \ref{orbifold_co} holds.  Since we also assume that $\pi_1(M)$ is infinite, the class of the fibre $h$ is of infinite order (See, e.g.  \cite[Proposition 4.1(1)]{BRW}).  Therefore from the short exact sequence
\[1 \rightarrow \langle h \rangle \rightarrow \pi_1(M) \rightarrow \pi_1^{orb}(\mathcal{B}) \rightarrow 1
\]
we can lexicographically construct the required circular ordering of $\pi_1(M)$ using Lemma \ref{orbifold_co} and Proposition \ref{sesprop}, completing the proof in these cases.

For the remaining cases, first suppose that $\mathcal{B} = \mathbb{S}^2(\a_1, \a_2, \dots, \a_n)$ and $\Sigma_{i=1}^n \frac{1}{\a_i}  \geq n-2$.  Note that necessarily $n \leq 4$, and our assumption that $\pi_1(M)$ is infinite and $b_1(M)=0$ rules out $n=0$, $n=1$, $n=2$.  When $n=3$ and $\frac{1}{\a_1}+\frac{1}{\a_2}+\frac{1}{\a_3} > 1$ the group $\pi_1(M)$ is finite, so we need not consider this case.  On the other hand when $\Sigma_{i=1}^n \frac{1}{\a_i}  = n-2$ then $\pi_1(M)$ has infinite abelianization (\cite[Proposition 2]{Hei1} and \cite[VI.13  Example]{Ja}), and so is left-orderable.  Last, if $M$ has base orbifold $\mathbb{P}^2$ or $\mathbb{P}^2(\a_1)$ then $\pi_1(M)$ is finite, $\mathbb{Z}$ or $\mathbb{Z}_2*\mathbb{Z}_2$ \cite[VI.11.(c)]{Ja}. Since we have assumed that $\pi_1(M)$ is infinite, $\pi_1(M)$ is either $\mathbb{Z}$ or $\mathbb{Z}_2*\mathbb{Z}_2$ which is circularly in both cases.   We construct a circular ordering of $\pi_1(M)$ for which $\mathrm{rot}(h)$ is zero in either case as follows: When $\pi_1(M)$ is $\mathbb{Z}$ we equip $\mathbb{Z}$ with a secret left-ordering, and when $\pi_1(M) \cong \mathbb{Z}_2*\mathbb{Z}_2 \cong \langle a, h \mid a^2=id, aha^{-1} = h^{-1}\rangle$ we circularly order $\pi_1(M)$ lexicographically using 
\[ 1 \rightarrow \langle h \rangle \rightarrow \pi_1(M) \rightarrow \mathbb{Z}/2\mathbb{Z} \rightarrow 1.
\]

We observe that (2) follows from the defining relations of the fundamental group $\pi_1(M)$.  When $M$ has nonorientable base orbifold and admits at least one exceptional fibre, then the relation $\gamma h \gamma^{-1} = h^{-1}$ holds in $\pi_1(M)$.  As rotation number is invariant under conjugation, we get that $\mathrm{rot}_c(h) = \mathrm{rot}_c(h^{-1})$ for every circular ordering $c$ of $\pi_1(M)$, from which it follows that $\mathrm{rot}_c(h) \in \{0, 1/2\}$.

To prove (3), suppose that $\pi_1(M)$ is left-orderable.   Using Proposition \ref{cofinal_ordering}, choose a left-ordering of $\pi_1(M)$ relative to which the class of the fibre is cofinal.  The result now follows from Proposition \ref{cofinal_central_gives_every_r}.

Finally, (4) follows from the short exact sequence of the fibration
\[ 1 \rightarrow \pi_1(F) \rightarrow \pi_1(M) \rightarrow \pi_1(\mathbb{S}^1) \rightarrow 1
\] 
and the observation that we may use any circular ordering of $\pi_1(\mathbb{S}^1) \cong \mathbb{Z}$ we please in applying Lemma \ref{sesprop}.
\end{proof}

\subsection{Graph manifolds}

We begin with a few preliminaries to establish notation and some well-known facts.  Recall that every compact, orientable, irreducible $3$-manifold $M$ admits a unique minimal family of disjoint incompressible tori $\mathcal{T}$ such that $M \setminus \mathcal{T}$ consists of Seifert fibred $3$-manifolds and atoroidal $3$-manifolds, called the JSJ decomposition.  By a graph manifold, we mean a compact, connected, orientable, irreducible $3$-manifold admitting a JSJ decomposition into Seifert fibred pieces.  By using the tori of the JSJ decomposition to cut our graph manifolds into Seifert fibred pieces, we know that the collection of such tori is minimal for each graph manifold, and each torus is incompressible and thus $\pi_1$-injective.

Note that it follows from Proposition \ref{COS} that graph manifolds with infinite first homology always have left-orderable fundamental group, in particular their fundamental groups are always circularly orderable.  Thus when it comes to circular orderability, we will only consider the case of rational homology sphere graph manifolds.  

\begin{lemma}\label{graph1}
 Let $W$ be a graph manifold with a torus boundary that is not homeomorphic to a Seifert fibred manifold. If $\alpha \in H_1(\partial W; \mathbb{Z})/\{ \pm 1\}$ is not the slope of a regular fibre, then $W(\a)$ is a graph manifold.
\end{lemma}
\begin{proof} Let $M_1, M_2, \dots ,M_p$ be the Seifert pieces of the JSJ-decomposition of $W$. Assume that the JSJ-decomposition of $W$ has only two pieces $M_1$ and $M_2$. Without loss of generality, assume that the torus boundary of $W$ is contained in $M_2$. Since $\a$ is not the slope of a regular fibre, $M_2(\a)$ is a Seifert fibred space with boundary $\partial M_2\setminus \partial W$, by \cite{Hei2} and \cite[Theorem 5.1]{DSh}. Therefore, $M_2(\a)$ is irreducible (because it is not $\mathbb{S}^1\times\S^2$ or $\mathbb{S}^1\tilde{\times}\S^2$ or $\mathbb{P}^3\#\mathbb{P}^3$ \cite[Proposition 4.1 (3)]{BRW}). Hence, $W(\a)$ is a graph manifold with Seifert pieces $M_1$ and $M_2(\a)$. Now assume that $p>2$. Let $M_k$ be the Seifert piece of $W$ containing the torus boundary of $W$. Since $\a$ is not the slope of a regular fibre, $M_k(\a)$ is a Seifert fibred space with boundary $\partial M_k\setminus \partial W$, by \cite{Hei2} and \cite[Theorem 5.1]{DSh}. Therefore, $M_k(\a)$ is irreducible (because it is not $\mathbb{S}^1\times\S^2$ or $\mathbb{S}^1\tilde{\times}\S^2$ or $\mathbb{P}^3\#\mathbb{P}^3$ \cite[Proposition 4.1 (3)]{BRW}).
Hence $W(\a)$ is a graph manifold with Seifert pieces $M_1$, $M_2,$ $\dots,$ $M_k(\a),$ $\dots,$ $M_p$.
\end{proof}

\begin{lemma}
\label{reg_fibre}
Suppose that $W$ is a graph manifold and that $\partial W$ is a torus.  Suppose $W$ has Seifert fibred pieces $M_1, \ldots, M_{\ell}$ and that $\partial W \subset M_1$.  Then if $W(\alpha)$ is reducible, $\alpha$ must be the slope of the regular fibre in a Seifert fibration of $M_1$ unless $W$ is a Seifert fibred space with base orbifold $\mathcal{B}_1(\alpha_1,\dots, \alpha_{s_1})$ such that $\alpha_1=\cdots =\alpha_{s_1}=1$ or $s_1=0$, and the geometric intersection number satisfies $\Delta(\a, h)=1$ where $h$ is the slope of a regular fibre.
\end{lemma}
\begin{proof}
Let $\alpha \in H_1(\partial W; \mathbb{Z})/\{ \pm 1\}$ be a slope. We have two cases:

{\bf Case 1:} Assume that $\alpha$ is not the slope of a regular fibre. Then the geometric intersection number satisfies $\Delta(\a, h)=d>0$, where $h$ is the slope of a regular fibre in $M_1$. Hence, if the base orbifold of $M_1$ is $\mathcal{B}_1(\alpha_1,\dots, \alpha_{s_1})$ then the base orbifold of $M_1(\a)$ is $\mathcal{B}'_1(\alpha_1,\dots, \alpha_{s_1}, d)$, where $\mathcal{B}'$ is obtained from $\mathcal{B}$ by filling a disk. We have two subcases:

{\bf Subcase 1:} Assume that $M_1=W$ that is $W$ is a Seifert fibred space. Then $W(\a)$ is a Seifert fibred space by \cite{Hei2} and \cite[Theorem 5.1]{DSh}. Hence $W(\a)$ is irreducible unless it is $\mathbb{S}^1\times \S^2$ or $\mathbb{P}^3\#\mathbb{P}^3$ \cite[Proposition 4.1 (3)]{BRW} (This is possible only in the case where $\alpha_1=\cdots =\alpha_{s_1}=1$ or $s_1=0$, and $d=1$.)

{\bf Subcase 2:} If $W$ is not a Seifert fibred space, then $W(\a)$ is a graph manifold by Lemma \ref{graph1}. Hence, $W(\a)$ is irreducible by definition of a graph manifold.

{\bf Case 2:} If $\a$ is the slope of a regular fibre then $W(\a)$ may be reducible \cite{Hei2} and \cite[Theorem 5.1]{DSh}.
\end{proof}

Given a rational homology sphere graph manifold $W$, note that the underlying graph must be a tree.  In this situation for a fixed JSJ torus $T$, note that $W \setminus T$ has two components.  We denote the closure of these components by $W_1$ and $W_2$, and $\phi: \partial W_1 \rightarrow \partial W_2$ the homeomorphism such that $W = W_1 \cup_{\phi} W_2$.  We use this fixed notation for the discussion and proof below.

Define a class $\mathcal{C}$ of rational homology sphere graph manifolds to be the smallest collection of $3$-manifolds satisfying:
\begin{enumerate}
\item All connected, irreducible rational homology sphere Seifert fibred manifolds belong to $\mathcal{C}$.
\item A rational homology sphere graph manifold $W$ is in $\mathcal{C}$ if and only if 
\begin{enumerate}
\item For every JSJ torus $T \subset W$, at least one of $W_1(\phi^{-1}_*(\lambda_{W_2}))$ and $W_2(\phi_*(\lambda_{W_1}))$ has infinite fundamental group, and 
\item For every JSJ torus $T \subset W$, every irreducible manifold of the form $W_i(\alpha)$ satisfying $|H_1(W_i(\alpha); \mathbb{Z})| < \infty$ lies in $\mathcal{C}$.
\end{enumerate}
\end{enumerate}

\begin{theorem}\label{graphCO}
If $W \in \mathcal{C}$ and $\pi_1(W)$ is infinite, then it is circularly orderable.
\end{theorem}
\begin{proof}

For every graph manifold $W$ let $n_W$ denote the minimal number of tori required to cut $W$ into Seifert fibred pieces (i.e., the number of tori in its JSJ decomposition).  If $n_W = 0$ and $\pi_1(W)$ is infinite, then it is circularly orderable by Theorem \ref{SFCO}.  For induction, assume that $k \geq 0$ and that $\pi_1(W)$ is circularly orderable for all $W \in \mathcal{C}$ with $n_W \leq k$, and consider the case of $n_W = k+1$.

For a manifold $W \in \mathcal{C}$ with $n_W = k+1$ and infinite fundamental group, choose a JSJ torus $T$ such that $W \setminus T$ results in two pieces $W_1$ and $W_2$, such that $W_1$ is Seifert fibred.  

First suppose that $W_2(\phi_*(\lambda_{W_1}))$ has infinite fundamental group and let $M_1, \dots, M_{\ell}$ denote the Seifert fibred pieces of the JSJ decomposition of $W_2$, assume that $\partial W_2 \subset M_1$.

Next suppose that $W_2(\phi_*(\lambda_{W_1}))$ is irreducible.  If $H_1(W_2(\phi_*(\lambda_{W_1})); \mathbb{Z})$ is finite, then we conclude $W_2(\phi_*(\lambda_{W_1})) \in \mathcal{C}$ by property 2(b).   Then $\phi_*(\lambda_{W_1})$ is not the slope of a regular fibre in the outermost piece of $M_1$ and the JSJ components of $W_2(\phi_*(\lambda_{W_1}))$ are precisely the Seifert fibred manifolds $M_1(\phi_*(\lambda_{W_1})), M_2, \dots, M_{\ell}$ by Lemma \ref{graph1}.  Thus $W_2(\phi_*(\lambda_{W_1}))$ is a graph manifold in $\mathcal{C}$ with $n_{W_2(\phi_*(\lambda_{W_1}))} < n_W$, so $\pi_1(W_2(\phi_*(\lambda_{W_1})))$ is circularly orderable by induction.  On the other hand if $H_1(W_2(\phi_*(\lambda_{W_1})); \mathbb{Z})$ is infinite then $\pi_1(W_2(\phi_*(\lambda_{W_1})))$ is left-orderable, hence circularly orderable, by \cite[Theorem 3.2]{BRW}.  In either case it follows that $\pi_1(W)$ is circularly orderable by Proposition \ref{rat_long_CO}(2).

Now suppose that $W_2(\phi_*(\lambda_{W_1}))$ is reducible, in which case $\phi_*(\lambda_{W_1})$ must be the slope of a regular fibre in $\partial M_1$ or  $W_2(\phi_*(\lambda_{W_1}))$ is $\mathbb{S}^1\times \S^2$ or $\mathbb{P}^3\#\mathbb{P}^3$ by Lemma \ref{reg_fibre}.  If $W_2(\phi_*(\lambda_{W_1}))$ is $\mathbb{S}^1\times \S^2$ or $\mathbb{P}^3\#\mathbb{P}^3$, then as each has circularly orderable fundamental group, $\pi_1(W)$ is circularly orderable by Proposition \ref{rat_long_CO}(2).

Next, in the case that $\phi_*(\lambda_{W_1})$ is the slope of a regular fibre in $\partial M_1$, recall that $H_1(W; \mathbb{Z})$ is finite and so the surface underlying the base orbifold of $M_1$ has genus zero.  Further suppose that $M_1$ has $r$ exceptional fibres and boundary tori $T, T_1, \dots, T_m$, and that $W_2 \setminus M_1$ has components $Y_1, \dots, Y_m$ where each $Y_j$ is a graph manifold with torus boundary.  Let $\phi_j : \partial Y_j \rightarrow T_j \subset M_1$ for $j = 1, \dots, m$ denote the gluing maps that recover $W_2$ from the pieces $M_1, Y_1, \dots, Y_m$.  In this case, by \cite[Theorem 5.1]{DSh} and \cite{Hei2}, filling $M_1$ along $\phi_*(\lambda_{W_1})$ yields
\[ M_1(\phi_*(\lambda_{W_1})) \cong L_1 \# \dots \# L_r \# (\mathbb{S}^1 \times \D^2) \# \dots \# (\mathbb{S}^1 \times \D^2) 
\]
or 
\[ M_1(\phi_*(\lambda_{W_1})) \cong L_1 \# \dots \# L_r \# (\mathbb{S}^1 \times \D^2) \# \dots \# (\mathbb{S}^1 \times \D^2) \#  (\mathbb{S}^1 \times \S^2) 
\]
depending on whether or not the underlying manifold of the base orbifold is a punctured $\S^2$ or $\P^2$.  Here, $L_1, \dots, L_r$ are lens spaces, there are $m$ copies of $\mathbb{S}^1 \times \D^2$ each arising from a torus component of of $\partial M_1 \setminus T$, and each $\partial \D^2$ is path homotopic to (the image of) a regular fibre in the Dehn filled manifold $M_1(\phi_*(\lambda_{W_1}))$.

Therefore if we denote the slope of a regular fibre on $\partial M_1$ by $h$, it follows that
\[ W_2(\phi_*(\lambda_{W_1})) \cong  L_1 \# \dots \# L_r \# Y_1((\phi_1^{-1})_*(h)) \# \dots \# Y_m((\phi_m^{-1})_*(h))
\]
or 
\[ W_2(\phi_*(\lambda_{W_1})) \cong  L_1 \# \dots \# L_r \# Y_1((\phi_1^{-1})_*(h)) \# \dots \# Y_m((\phi_m^{-1})_*(h))\# (\mathbb{S}^1 \times \S^2),
\]
again depending on whether or not the base orbifold is orientable.

In either case, we may proceed as follows.  Suppose that every manifold $Y_j((\phi_j^{-1})_*(h))$ for $j =1 , \dots , m$ has infinite fundamental group.  If $Y_j$ is homeomorphic to a Seifert fibred manifold, it follows that its fundamental group of $Y_j((\phi_j^{-1})_*(h))$ is circularly orderable by Theorem \ref{SFCO}.  On the other hand if $Y_j$ is not homeomorphic to a Seifert fibred manifold, then as $(\phi_j^{-1})_*(h)$ is not the slope of a regular fibre in the outermost Seifert fibred piece of $Y_j$, we know that $Y_j((\phi_j^{-1})_*(h))$ is irreducible by Lemma \ref{graph1}.   Thus either $|H_1(Y_j((\phi_j^{-1})_*(h)); \mathbb{Z})| = \infty$ or $|H_1(Y_j((\phi_j^{-1})_*(h)); \mathbb{Z})| < \infty$ and $Y_j((\phi_j^{-1})_*(h)) \in \mathcal{C}$ by property 2(b).  In the former case, the fundamental group of $Y_j((\phi_j^{-1})_*(h))$ is circularly orderable since it is in fact left-orderable by Theorem \cite[Theorem 3.2]{BRW}.  In the latter case we may argue that $n_{Y_j((\phi_j^{-1})_*(h))} < n_W$ so the fundamental group of $Y_j((\phi_j^{-1})_*(h))$ is circularly orderable by induction.

It follows that $\pi_1( W_2(\phi_*(\lambda_{W_1})))$ is a free product of circularly orderable groups, and thus is circularly orderable.  Note it is also infinite by assumption.  Last, note that $\lambda_{W_1}$ is not the slope of a regular fibre in $\partial W_1$, since it is glued via $\phi_*$ to the slope of a regular fibre in $\partial M_1$, and thus $W_1(\lambda_{W_1})$ is irreducible by \cite{Hei2}. Therefore $\pi_1(W)$ is circularly orderable by Proposition \ref{rat_long_CO}(2). 

Now suppose that there exists $j_0$ such that the fundamental group of $Y_{j_0}((\phi_{j_0}^{-1})_*(h))$ is finite, in which case $Y_{j_0}$ is Seifert fibred and $(\phi_{j_0}^{-1})_*(h)$ is not equal to the rational longitude $\lambda_{Y_{j_0}}$ of $Y_{j_0}$.  One of the components of $W \setminus T_{j_0}$ is $Y_{j_0}$, we will call the other component $W''$, so that $W = Y_{j_0} \cup_{\phi_{j_0}} W''$.  Note that $W''((\phi_{j_0})_*(\lambda_{Y_{j_0}}))$ is irreducible by Lemma \ref{graph1}, since $(\phi_{j_0})_*(\lambda_{Y_{j_0}})$ is not the slope a regular fibre in the outermost piece of $W''$.  Thus either the homology of $W''((\phi_{j_0})_*(\lambda_{Y_{j_0}}))$ is infinite and so it has left-orderable fundamental group, or the homology is finite and so $W''((\phi_{j_0})_*(\lambda_{Y_{j_0}})) \in \mathcal{C}$.  In the latter case $W''((\phi_{j_0})_*(\lambda_{Y_{j_0}}))$ has infinite fundamental group since it is an irreducible graph manifold whose JSJ decomposition consists of two or more Seifert fibred pieces, and its fundamental group is circularly orderable by induction. Thus we may apply Proposition \ref{rat_long_CO}(2) with $Y_{j_0}$ in place of $M_1$ and $W''$ in place of $M_2$ in order to conclude that $\pi_1(W)$ is circularly orderable. 

Last, suppose $W_2(\phi_*(\lambda_{W_1}))$ has finite fundamental group (and thus $W_2$ is Seifert fibred) so that $W_1(\phi^{-1}_*(\lambda_{W_2}))$ is infinite by property (2)(a).  Then $\pi_1(W_1(\phi^{-1}_*(\lambda_{W_2})))$ is circularly orderable by Theorem \ref{SFCO}. Thus the result follows from Proposition \ref{rat_long_CO}(2).
\end{proof}


We can more precisely codify the manifolds covered by Theorem  \ref{graphCO} as follows. The possible base orbifolds of a Seifert fibred manifold admitting a single incompressible torus boundary component and a finite filling are:
$$\mathcal{A}:=\{ \mathbb{D}^2(p,q),\; \mathbb{D}^2(2,2,r),\\
 \mathbb{D}^2(2,3,3), \; \mathbb{D}^2(2,3,4),\; \mathbb{D}^2(2,3,5),\; {\rm \; with }\; r\geq 1, p\geq 2, q\geq 2\}.$$
 
 \begin{proof}[Proof of Theorem \ref{special graph case}]
 First, if $W$  is not a rational homology sphere then the first Betti number $b_1(W)$ is positive, so $\pi_1(W)$ is left-orderable by \cite[Theorem 3.2]{BRW}.  Second to this, if $W$ is a graph manifold satisfying the assumptions of Theorem \ref{special graph case}, and if $W$ is Seifert fibred, then Theorem \ref{SFCO} finishes the proof. 

From here, we complete the proof by showing that if $W$ is a rational homology sphere graph manifold satisfying the hypotheses of Theorem \ref{special graph case}, then $W \in \mathcal{C}$ so Theorem \ref{graphCO} applies.

To do this, we induct on the number $n_W$ of tori in the JSJ decomposition of $W$, noting that if $W$ satisfies the hypotheses of Theorem \ref{special graph case} and $n_W = 0$ then $W \in \mathcal{C}$ by definition. Next, suppose that every rational homology sphere graph manifold $W$ satisfying the hypotheses of Theorem \ref{special graph case} with $n_W <k$ lies in $\mathcal{C}$, and consider $W$ with $n_W = k$.  

Suppose $W$ has Seifert fibred pieces $M_1, \ldots, M_{\ell}$. Choose an arbitrary JSJ torus $T \subset W$ and cut along $T$ to arrive at $W = W_1 \cup_{\phi} W_2$.

Considering $W_1(\phi^{-1}_*(\lambda_{W_2}))$ and $W_2(\phi_*(\lambda_{W_1}))$, if either admits a JSJ decomposition with one or more JSJ tori, then its fundamental group is infinite.  On the other hand, if either is Seifert fibred then again the fundamental group is infinite since the hypotheses of Theorem \ref{special graph case} imply the base orbifold of $M_i$ is not in $\mathcal{A}$ for $i = 1, \dots, \ell$.  In any event, $W$ satisfies property 2(a) in the definition of $\mathcal{C}$.

Next, suppose that $W_1(\alpha)$ is irreducible for some $\alpha \in H_1(W_1; \mathbb{Z})/\{ \pm 1\}$ and that $H_1(W(\alpha); \mathbb{Z})$ is finite.  If $W_1$ is Seifert fibred then $W_1(\alpha)$ is Seifert fibred and so lies in $\mathcal{C}$ by definition. On the other hand if $W_1$ is not Seifert fibred then by Lemma \ref{reg_fibre}, $\alpha$ must not be the slope of the regular fibre in the outermost piece of $W_1$ and so $W_1(\alpha)$ is a graph manifold by Lemma \ref{graph1}.  Moreover $n_{W_1(\alpha)} <k$ and so  $W_1(\alpha) \in \mathcal{C}$ by our induction assumption.    As the same arguments hold for $W_2$, we conclude that $W$ satisfies property 2(b) in the definition of $\mathcal{C}$, and thus $W \in \mathcal{C}$.
 \end{proof}
  
  In fact, suppose $W$ is an arbitrary rational homology sphere graph manifold admitting a JSJ decomposition $W = M_1 \cup_{\phi}M_2$ such that the base orbifold of $M_i$ lies in $\mathcal{A}$.  If $\pi_1(W)$ is circularly orderable for every such $W$, then Theorem \ref{graphCO} and its proof can be used, \textit{mutatis mutandis}, to show that the fundamental group of \emph{every} rational homology sphere graph manifold is circularly orderable whenever it is infinite.  We thus make explicit exactly which graph manifolds having two pieces in their JSJ decomposition are covered by Theorem \ref{graphCO}.
  
Let $\mathcal{E}$ be the set of graph manifolds whose JSJ decomposition has at least one Seifert piece with base orbifold in $\mathcal{A}$.    Further, set $\mathcal{F}=\{\mathbb{D}^2(2,2), \mathbb{D}^2(2,3), \mathbb{D}^2(3,3), \mathbb{D}^2(3,4), \mathbb{D}^2(3,5) \}\subset \mathcal{A}$.  
  
\begin{corollary}
\label{2piecescase}
Let $W$ be a rational homology sphere graph manifold whose JSJ decomposition has only two Seifert pieces $M_1$ and $M_2$ with base orbifolds $\mathcal{B}_1(p_1^1,\dots, p_{s_1}^1)$ and $\mathcal{B}_2(p_1^2,\dots, p_{s_2}^2)$ respectively.  Let $\phi : \partial M_1 \rightarrow \partial M_2$ denote the gluing map that recovers $W$ from the pieces $M_1$ and $M_2$, and let $\lambda_i$ and $h_i$ denote the rational longitude and the slope of a regular fibre on $\partial M_i$ for each of $i = 1, 2$. Suppose that $W$ satisfies:
\begin{enumerate}
\item $W\notin \mathcal{E}$; or
    \item the base orbifold of $M_1$ lies in $\mathcal{A}$ and the base orbifold of $M_2$ is not in $\mathcal{A}$; or
    \item the base orbifolds of both $M_1$ and $M_2$ lie in $\mathcal{A}$, and 
    \begin{enumerate}
        \item if $\mathcal{B}_1(p_1^1,\dots, p_{s_1}^1) \notin \mathcal{F}$ and $\mathcal{B}_2(p_1^2,\dots, p_{s_2}^2)\notin \mathcal{F}$, then either $\Delta(\phi_*(\lambda_1), h_2)> 1$ or $\Delta(\phi^{-1}_*(\lambda_2), h_1)>1$;
     \item if $\mathcal{B}_1(p_1^1,\dots, p_{s_1}^1)\notin\mathcal{F}$ and $\mathcal{B}_2(p_1^2,\dots, p_{s_2}^2)\in \mathcal{F}$, then $\Delta(\phi^{-1}_*(\lambda_2), h_1)> 1$, or $\mathcal{B}_2(p_1^2,\dots, p_{s_2}^2)=\mathbb{D}^2(2,3)$ and $\Delta(\phi_*(\lambda_1), h_2)> 5$, or $\mathcal{B}_2(p_1^2,\dots, p_{s_2}^2)=\mathbb{D}^2(3,3)$ and $\Delta(\phi_*(\lambda_1), h_2)> 2$, or $\mathcal{B}_2(p_1^2,\dots, p_{s_2}^2)=\mathbb{D}^2(3,4)$ and $\Delta(\phi_*(\lambda_1), h_2)> 2$, or $\mathcal{B}_2(p_1^2,\dots, p_{s_2}^2)=\mathbb{D}^2(3,5)$ and $\Delta(\phi_*(\lambda_1), h_2)> 2$;
    \end{enumerate}
\end{enumerate}
then $\pi_1(W)$ is circularly orderable.
\end{corollary}
\begin{proof}
We prove only case (3)(b), as the other statements claimed are all consequence of Theorem \ref{graphCO}, since any such manifold lies in $\mathcal{C}$.

Assume that  $\mathcal{B}_1(p_1^1,\dots, p_{s_1}^1)\notin\mathcal{F}$, $\mathcal{B}_2(p_1^2,\dots, p_{s_2}^2)\in \mathcal{F}$ and $\Delta(\phi^{-1}_*(\lambda_2), h_1)> 1$. Let $\varphi=\phi^{-1}$. Since $\mathcal{B}_1(p_1^1,\dots, p_{s_1}^1)\notin\mathcal{F}$ and $\Delta(\phi^{-1}_*(\lambda_2), h_1)>1$, $\pi_1(M_1(\varphi(\lambda_2)))$ is infinite and circularly orderable. Since $\mathcal{B}_2(p_1^2,\dots, p_{s_2}^2)\in \mathcal{F}$, $M_2$ is Seifert fibred with incompressible boundary. Therefore, $\pi_1(M)$ is circularly orderable by Proposition \ref{rat_long_CO}.

Next assume that $\mathcal{B}_2(p_1^2,\dots, p_{s_2}^2)=\mathbb{D}^2(2,3)$ (resp. $\mathbb{D}^2(3,3),\mathbb{D}^2(3,4),  \mathbb{D}^2(3,5)$ ) and also that $\Delta(\phi_*(\lambda_1), h_2)> 5$ (resp. $\Delta(\phi_*(\lambda_1), h_2)>2$). Then, the base orbifold of $M_2(\phi(\lambda_1))$ is $\mathbb{S}^2(2, 3, a)$ (resp. $\mathbb{S}^2(3, 3, a)$, $\mathbb{S}^2(3, 4, a)$, $\mathbb{S}^2(3, 5, a)$) where $a \geq 6$ (resp. $a \geq 3$). Hence, $\pi_1(M_2(\phi(\lambda_1)))$ is infinite and circularly orderable. Since $\mathcal{B}_1(p_1^1,\dots, p_{s_1}^1)\in \mathcal{A}$, $M_1$ is Seifert fibred with incompressible boundary. Therefore, $\pi_1(M)$ is circularly orderable by Proposition \ref{rat_long_CO}.
%
%
%
\end{proof}


\subsection{Graph manifolds with two pieces}

Our goal in this section is to show that we can circularly order many of the fundamental groups of graph manifolds having two pieces in their JSJ decomposition, that are not covered by Corollary \ref{2piecescase}.

\begin{lemma}\label{cyclic}
 Let $W$ be a 3-manifold obtained by gluing two knot exteriors in some integer homology 3-spheres, on their torus boundary by some orientation reversing homeomorphism. Then $H_1(W, \Z)$ is cyclic.
\end{lemma}
\begin{proof}
Let $K_1$ and $K_2$ be two knots in some integer homology 3-spheres $W_1$ and $W_2$ respectively. Let $M_1$ and $M_2$ be the knot exteriors of $K_1$ and $K_2$ respectively. Let $W$ be the $3$-manifold obtained by identifying $\partial M_1$ and $\partial M_2$ by some orientation reversing homeomorphism $\varphi$. Let $(\mu_1, \lambda_1)$ and $(\mu_2, \lambda_2)$ be the meridian-longitude slope pairs of $\partial M_1$ and $\partial M_2$ respectively. Let $\varphi_*(\mu_1)=a\mu_2 + b\lambda_2$ and $\varphi_*(\lambda_1)=c\mu_2 + d\lambda_2$.

Let $T=\partial M_1\cong \partial M_2$. We have the Mayer-Vietoris sequence:

$\cdots \longrightarrow H_1(T; \mathbb{Z})\stackrel{\phi_1}{\longrightarrow} H_1(M_1; \mathbb{Z})\oplus H_1(M_2; \Z)\stackrel{\psi_1}{\longrightarrow} H_1(W; \mathbb{Z})\stackrel{\partial_1}{\longrightarrow} H_0(T; \mathbb{Z})\stackrel{\phi_0}{\longrightarrow} H_0(M_1; \mathbb{Z})\oplus H_0(M_2; \Z)\stackrel{\psi_0}{\longrightarrow}
 H_0(W; \mathbb{Z})\longrightarrow 0.$
 
We have that $\lambda_1=0$ in $H_1(M_1; \Z)$, and $H_1(M_1; \Z)$ is generated by $\mu_1.$ We have also that $\lambda_2=0$ in $H_1(M_2; \Z)$, and $H_1(M_2; \Z)$ is generated by $\mu_2.$ We consider $\{\mu_1, \lambda_1\}$ to be the generators of $H_1(T; \Z)$. Hence, $\phi_1(\mu_1)=i_T(\mu_1)\oplus -i_T(\mu_1)=(\mu_1, -a\mu_2)$ and $\phi_1(\lambda_1)=i_T(\lambda_1)\oplus -i_T(\lambda_1)=(0, -c\mu_2)$. By exactness we have that $im(\phi_1)=ker(\psi_1)$, $im(\psi_1)=ker(\partial_1)$, and $im(\partial_1)=ker(\phi_0)$. Since $T$, $M_1$ and $M_2$ are connected, hence path connected, $\phi_0$ is injective. Therefore, $im(\partial_1)=0$ and $im(\psi_1)=ker(\partial_1)=H_1(W; \Z).$ So, $H_1(W; \Z)=im(\psi_1)\cong H_1(M_1; \mathbb{Z})\oplus H_1(M_2; \Z)/ker(\psi_1)=H_1(M_1; \mathbb{Z})\oplus H_1(M_2; \Z)/im(\phi_1)=\Z\oplus\Z/im(\phi_1).$ We have that $im(\phi_1)=\langle (\mu_1, -a\mu_2), (0, -c\mu_2)\rangle\cong \langle (1, -a), (0, -c)\rangle$. Consider the matrix $\big(\begin{smallmatrix}
  1 & 0\\
  -a & -c
\end{smallmatrix}\big)$. By adding $a$ times the first row to the second row we obtain  $\big(\begin{smallmatrix}
  1 & 0\\
  0 & -c
\end{smallmatrix}\big)$. Hence $im(\phi_1)=\langle (\mu_1, -a\mu_2), (0, -c\mu_2)\rangle\cong \langle (1, -a), (0, -c)\rangle\cong \langle (1,0), (0, -c)\rangle\cong\Z\oplus |c|\Z$, and so $$H_1(W; \Z)=im(\psi_1)\cong H_1(M_1; \mathbb{Z})\oplus H_1(M_2; \Z)/ker(\psi_1)=H_1(M_1; \mathbb{Z})\oplus H_1(M_2; \Z)/im(\phi_1),$$
with this final group being isomorphic to $\Z\oplus\Z/(\Z\oplus |c|\Z) \cong \Z_{|c|}.$
\end{proof}
If $W$ is a rational homology sphere graph manifold, we can construct a graph called the {\it splice diagram} $\Gamma(W)$ as follows: nodes are in one to one correspondence with the Seifert pieces of $W$.  Two nodes are connected by an edge if the corresponding Seifert fibered pieces are glued together along a common torus boundary component.  To each node, one attaches a leaf for each singular fibre of the corresponding Seifert fibred piece. 

We then decorate the graph so constructed as follows: each edge corresponding to a leaf is labeled with the multiplicity of the corresponding singular fibre. Let $v$ be a node of $\Gamma(W)$ and $e$ be an edge of $\Gamma(W)$ connecting $v$. Let $N$ and $K$ be the two pieces of $W$ obtained by cutting $W$ along the torus corresponding to $e$ such that $K$ does not contain $M_v$, where $M_v$ is the Seifert piece of $W$ corresponding to $v$. Let $D$ be the manifold obtained by Dehn filling $K$ with a solid torus $\mathbb{D}^2\times \mathbb{S}^1$ by identifying a regular fibre of $\partial K$ with a meridian of $\mathbb{D}^2\times \mathbb{S}^1$. Let $d_v=|H_1(D)|$, and take $d_v$ to be the label of the edge $e$ at the node $v$. We also decorate the nodes of $\Gamma(W)$ with signs $+$ or $-$ corresponding to the sign of the linking number of two non-singular fibers in the Seifert fibration (see \cite[Section 2]{Pe} for more details).

\begin{figure}[H]
\centering
\def\svgwidth{0.65\columnwidth}
 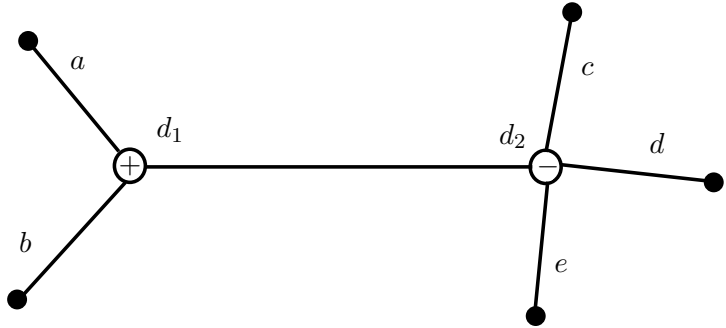
 \caption{Example of a splice diagram.}
\label{splice}
\end{figure}
\begin{proposition}\label{torusknot}
Suppose that $M_1$ and $M_2$ are Seifert fibred spaces with incompressible boundaries and base orbifolds $\mathbb{D}^2(a_1, \dots, a_s)$ and $\mathbb{D}^2(b_1, \dots, b_t)$ respectively, and that each is the exterior of a knot in an integer homology sphere.
 Let $W$ be a graph manifold obtained by gluing $M_1$ and $M_2$ along their torus boundaries by some orientation reversing homeomorphism. If $gcd(a_i, a_l)=1$, $gcd(b_j, b_k)=1$ $(1\leq i \neq l\leq s,\; 1\leq j\neq k \leq t)$ and $gcd(a_i, b_k) = 1$ for all $i = 1, \dots, s$ and $k = 1, \dots t$ then $\pi_1(W)$ is circularly orderable.
\end{proposition}
\begin{proof}
With the restrictions on $a_i$ and $b_k$ as in the statement of the theorem, the manifold $W$ has a corresponding splice diagram $\Gamma(W)$ with two nodes.  The edge labels around any node in the diagram are all pairwise coprime, so by \cite[Corollary 6.3]{Pe} the universal abelian cover of $W$ is an integer homology sphere graph manifold. Hence, the commutator group $[\pi_1(W), \pi_1(W)]$ is left-orderable by \cite{CLW}. Since $H_1(W)$ is cyclic by Lemma \ref{cyclic}, $\pi_1(W)$ is circularly orderable by Proposition \ref{sesprop}. 
\end{proof}

In particular, Proposition \ref{torusknot} applies to manifolds $W = M_1 \cup_{\phi} M_2$ where $M_i$ are torus knot exteriors whose cone points have relatively coprime orders.  Manifolds of this form are not completely covered by Corollary \ref{2piecescase} or Theorem \ref{graphCO}.  We can further deal with other special cases of interest not covered by these theorems.

\begin{proposition}
\label{ibundle_prop}
Suppose that $W = M_1 \cup_{\phi} M_2$ where each $M_i$ is a twisted $I$-bundle over the Klein bottle.  Then $\pi_1(W)$ is circularly orderable.
\end{proposition}
\begin{proof}
The fundamental group of $W$ is an amalgamated free product of two Klein bottle groups $K_1=\langle a, b\;|\; a^2=b^2\rangle$ and $K_2=\langle c, d\;|\; c^2=d^2\rangle$, whose peripheral subgroups are $\langle a^2, ab \rangle$ and $\langle c^2, cd\rangle$ respectively.  The amalgamation is with respect to an isomomorphism
$\phi :\langle a^2, ab \rangle\longrightarrow \langle c^2, cd\rangle$. Observe that $K_1$ admits a lexicographic circular ordering arising from the short exact sequence
\[ 1 \rightarrow \langle a^2, ab \rangle \rightarrow K_1 \rightarrow \mathbb{Z} /2 \mathbb{Z} \rightarrow 1,
\]
and that $K_2$ admits a similar lexicographic circular ordering, with the choice of left-ordering on the subgroup $\langle a^2, ab \rangle$ (resp. $\langle c^2, cd\rangle$) being arbitrary.  As such, we can construct circular orderings of $K_1$ and $K_2$ so that the homomorphism $\phi$ is order preserving, and the subgroups $\langle a^2, ab \rangle$, $\langle c^2, cd \rangle$ are convex.\footnote{Convexity here means that the quotient group inherits a circular ordering.  For a more general definition and discussion of convexity in the context of circular orderings, see \cite{CG}.} Then by \cite[Proposition 1.1]{CG}, $\pi_1(W)$ is circularly orderable.

\end{proof}

\begin{proof}[Proof of Theorem \ref{thmSol}]
Suppose $M$ is such a manifold. By \cite[Theorem 1.7 (1)]{BRW}, if the boundary of $M$ is not empty, or $M$ is non-orientable, or $M$ is a torus bundle over the circle, then $\pi_1(M)$ is left-orderable.  Thus $\pi_1(M)$ is circularly orderable.  By \cite[Section 9]{BRW}, the only case which is left to check is when $M$ is orientable and the union of two twisted $I$-bundles over the Klein bottle $K$, which are glued together along their torus boundaries.  Therefore Proposition \ref{ibundle_prop} finishes the proof.
\end{proof}

It seems, however, that the special case of a graph manifold consisting of two Seifert fibred pieces, each admitting finite fillings, is out of reach of our current technology.  A general notion of ``slope detection by a circular ordering" is likely needed to deal with these last few cases, though our results thusfar contribute ample evidence for the truth of the following conjecture.

\begin{conjecture}
\label{graph conjecture}
Suppose $W$ is a rational homology sphere graph manifold.  If $\pi_1(W)$ is infinite, then it is circularly orderable. 
\end{conjecture}

\section{cyclic branched covers and Dehn surgery}\label{sec3}

With respect to certain well-known geometric constructions, left-orderability is conjectured to exhibit certain predictable behaviours.  In this section we contrast the conjectured ``predictable behaviours" of left-orderability with the behaviour of circular orderability, which is strikingly different and whose expected behaviour at this time is completely unknown.

\subsection{Cyclic branched covers} We recall the standard construction of the cyclic covers and cyclic branched covers of a knot in $\mathbb{S}^3$ in order to establish notation.

Let $K$ be an oriented knot in $\mathbb{S}^3$. Let $M_K$ be the exterior of $K$ and $S$ be a Seifert surface for $K$.
Isotope $S$ so that $S\cap\partial M_K$ is a longitude of $K$ and
let $F=S\cap M_K$. Let $C$ be a tubular neighborhood of $F$ in $M_K$, so that $C$ is homeomorphic to $F\times [-1,1]$.

Set $Y=M_K\setminus F \times (-1,1)$.  The boundary of $Y$ contains two copies of $F$, which we denote by
${F}^{-}\cong F\times \{-1\}$ and ${F}^+\cong F\times \{1\}$ of $F$, use these to create a triple $(Y,F^+,F^-)$. Consider $n$ copies of 
this triple, denoted by $(Y_i,F_i^+,F_i^-)$, $i=0,\dots, n-1$, and glue them together by identifying $F_0^+\subset Y_0$ with
$F_1^-\subset Y_1$, $F_1^+\subset Y_1$ with $F_2^-\subset Y_2$, $\dots$,  $F_{n-2}^+\subset Y_{n-2}$ with  
$F_{n-1}^-\subset Y_{n-1}$ and $F_{n-1}^+\subset Y_{n-1}$ with $F_0^-\subset Y_0$. Call the resulting space $X_n$. 

There is a regular covering map $g:X_n\longrightarrow M_K$ and its 
group of deck transformations is isomorphic to $\mathbb{Z}/n\mathbb{Z}$. 
The manifold $X_n$ is called the $n$-fold cyclic cover of $M_K$ and its 
fundamental group is isomorphic to $ker(\pi_1(M_K)\longrightarrow \mathbb{Z}/n\mathbb{Z})$. 
To construct the $n$-fold cyclic branched
cover $\Sigma_n(K)$, we glue a solid torus $V\cong D^2\times \mathbb{S}^1$ to $Y_n$ by identifying the meridian
$\partial D^2\times \{1\}$ of $V$ with the preimage of the meridian $\mu$ of $\partial M_K$  under the map $g:X_n\longrightarrow M_K$.
The manifold $\Sigma_n(K)$ that results is a closed, oriented $3$-manifold. 
For any $n\in \N$, let $q_n: X_n\longrightarrow \Sigma_n(K)$ be the inclusion map. The map $q_n$ induces a homomorphism $(q_n)_*:  \pi_1(X_n)\longrightarrow \pi_1(\Sigma_n(K))$ and ${\rm ker}(q_n)_*=\langle\langle \mu^n \rangle\rangle.$ Therefore we have a short exact sequence $1\longrightarrow \langle\langle \mu^n \rangle\rangle\longrightarrow\pi_1(X_n)\longrightarrow \pi_1(\Sigma_n(K))\longrightarrow 1$, which identifies the fundamental group of $\Sigma_n(K)$ as the quotient of a certain subgroup of the knot group $\pi_1(M_K)$. 

If $L\subset \mathbb{S}^3$ is an oriented link, then the $n$-fold cyclic branched cover of $L$, $\Sigma_n(L)$, can be also constructed see \cite{BBG}.

Set $$\mathcal{LO}_{br}(K) = \{ n \geq 2 \mid \pi_1(\Sigma_n(K)) \mbox{ is left-orderable} \}$$ and $$\mathcal{CO}_{br}(K) = \{ n \geq 2 \mid \pi_1(\Sigma_n(K)) \mbox{ is circularly orderable} \}.$$ Note that $\mathcal{LO}_{br}(K) \subset \mathcal{CO}_{br}(K)$.

Motivated by the L-space conjecture, in \cite[Question 1.8]{Tu} and \cite{BBG}, the authors ask whether or not the set $\mathcal{LO}_{br}(K)$ is always of the form $\{ n \mid n \geq N\}$ for some $N \geq 2$, or empty.  In contrast, circular orderability does not behave this way, with this behaviour being evident upon examining the torus knots.  For example, considering the trefoil, we have the following.

\begin{proposition}
\label{trefoil prop}
With notation as above, $\mathcal{CO}_{br}(3_1) = \{2 \} \cup \{ n \mid n \geq 6\}$.
\end{proposition}
\begin{proof}
The double branched cover of $3_1$ is the lens space $L(3,1)$, so its fundamental group is circularly orderable.  On the other hand, the $3, 4$, and $5$-fold branched cyclic covers of $3_1$ have fundamental group the quaternion group, the binary tetrahedal group, and the binary icosahedral group respectively, all of which are finite and non-cyclic.  It follows from Proposition \ref{finite case} that none of these groups are circularly orderable.  For $n\geq 6$ we know that $n \in \mathcal{LO}_{br}(3_1) \subset \mathcal{CO}_{br}(3_1)$ by \cite[Theorem 1.2(i)]{GL}.
\end{proof}

This behaviour is not confined to torus knots.

\begin{proposition}
\label{52 prop}
 With notation as above, 
$$ 2 \in \mathcal{CO}_{br}(5_2) \mbox{ and } 3 \notin \mathcal{CO}_{br}(5_2), \mbox{ and } \{ n \mid n \geq 9 \}\subset  \mathcal{CO}_{br}(5_2).$$

\end{proposition}
\begin{proof}
The knot $5_2$ is a two-bridge knot, corresponding to the fraction $7/4$.  As such, its double branched cover is the lens space $L(7,4)$ having fundamental group $\mathbb{Z}/7\mathbb{Z}$ and is therefore circularly orderable.

On the other hand, the fundamental group of the Weeks manifold $W$ is not circularly orderable by \cite[Theorem 9.2]{CaD}, yet $W$ is homeomorphic to $\Sigma_3(5_2)$ by \cite[Main result]{MeV}.

Last, $n \in \mathcal{LO}_{br}(5_2) \subset \mathcal{CO}_{br}(5_2)$ for all $n \geq 9$ by \cite{Hu}.
\end{proof}

There are also examples of knots for which $\mathcal{LO}_{br}(K)$ is empty.  Notable examples are the two-bridge knots $p/q = 2m+\frac{1}{2k}$, or $L_{[2k, 2m]}$ $(k, m>0)$ in Conway's notation \cite{DPT}.  However, for these knots $\mathcal{CO}_{br}(K)$ is never empty, because the double branched cover is the lens space $L(p,q)$ for which $\pi_1(L(p,q)) = \mathbb{Z}/p\mathbb{Z}$, and is thus circularly orderable.   Indeed, for the figure eight knot $4_1$, which is $L_{[2,2]}$, we can show that infinitely many of the cyclic branched covers have circularly orderable fundamental group.  We first require a proposition.

\begin{proposition} 
\label{divisible covers}
Let $K$ be a prime knot in $\mathbb{S}^3$. If $n \geq2$ and $\pi_1(\Sigma_n(K))$ is circularly orderable and infinite, then $\pi_1(\Sigma_m(K))$ is circularly orderable for all $m$ divisible by $n$.
\end{proposition}
\begin{proof}
By \cite[Lemma 2.11]{GL}, there exists a surjective group homomorphism $$q_{m,n} : \pi_1(\Sigma_m(K)) \longrightarrow \pi_1(\Sigma_n(K))$$ for any positive integer $m$ divisible by $n$.  By \cite{Plo}, the manifolds $\Sigma_m(K)$ are irreducible and so Proposition \ref{COS}(2) applies, we conclude that $\pi_1(\Sigma_m(K))$ is circularly orderable.
\end{proof}

\begin{proposition}
If $n$ is divisible by $3$, then $\pi_1(\Sigma_n(4_1))$ is circularly orderable.  In particular,  $ 3 \mathbb{N} \subset \mathcal{CO}_{br}(4_1)$.
\end{proposition}
\begin{proof}
The manifold $\Sigma_3(4_1)$ is homeomorphic to the Hantzsche-Wendt manifold, which is a Seifert fibred manifold with infinite fundamental group.  Therefore $\pi_1(\Sigma_3(4_1))$ is circularly orderable by Theorem \ref{SFCO}.  The result now follows from Proposition \ref{divisible covers}.
\end{proof}

These observations naturally lead to the following questions.

\begin{question}
What subsets of $\mathbb{N}$ can occur as $\mathcal{CO}_{br}(K)$ for a knot $K$ in $\mathbb{S}^3$?
\end{question}

\begin{question}
Is there a knot $K$ in $\mathbb{S}^3$ for which $\mathcal{CO}_{br}(K) = \emptyset$?
\end{question}

\subsection{Double branched covers}\label{sec4}
In this section we study the circularly orderability of the double branched cover of links, in particular, the case of double branched covers of alternating links. We start with the the observation that, if the L-space conjecture is true, then the fundamental group of the double branched cover of a quasi-alternating knot is never left-orderable (the double branched cover of a quasi-alternating link is an L-space \cite{OSz}).  

In contrast to this, there are alternating Montesinos links whose double branched cover are prism manifolds by \cite{JR, BHMNOV}, so their fundamental groups are not circularly orderable by Proposition \ref{finite case}.  Similarly, it turns out that the Weeks manifold is homeomorphic to $\Sigma_2(9_{49})$ \cite[Main result]{MeV}, and so the double branched cover of $9_{49}$ has non-circularly orderable fundamental group.  Yet $9_{49}$ is a quasi-alternating knot \cite{Jab}.

On the other hand, there are many examples of alternating and quasi-alternating knots whose double branched covers do have circularly orderable fundamental groups, as the next two results show.

\subsubsection{Generalized Fibonacci groups} Let $M(k, m)$ denote the double branched cover of the alternating link which is the closure of the $3$-strand braid $(\s_1^k\s_2^{-k})^m$, where $m$ and $k$ are positive integers. By \cite[p.169]{MaR}, the fundamental group of $M(k, m)$ is isomorphic to the Generalized Fibonacci group 
$$F^k_{2m}=\langle x_1, \dots, x_{2m}\; | \; x_ix_{i+1}^k=x_{i+2}, \; {\rm for \; any}\; i=1,\dots, 2m \rangle,$$
where the indices  are taken mod $2m$.
\begin{proposition}
For any $k\geq 2$, the fundamental group of $M(k,m)$ is circularly orderable.
\end{proposition}
\begin{proof}
For any $k\geq 2$, $M(k, m)$ is irreducible by \cite[Lemma 6 and p.171]{MaR}. Moreover we can define an epimorphism $\r: F_{2m}^k\longrightarrow \mathbb{Z}_k\ast\mathbb{Z}_k$ by $\r(x_{2i})\longmapsto x$ and $\r(x_{2i+1})\longmapsto y$, where $x$ and $y$ are the generators of $\mathbb{Z}_k\ast\mathbb{Z}_k$. Since $\mathbb{Z}_k\ast\mathbb{Z}_k$ is circularly orderable by Proposition \ref{free product}(2), by Proposition \ref{COS}(2) $\pi_1(M(k, m))$ is circularly orderable.
\end{proof}
\begin{remark}
 For $k\geq 2$ and $m=2$, the manifolds $M(k,2)$ are obtained by identifying two Seifert fibred spaces along a common torus boundary \cite[p.171]{MaR}. Thus $M(k,2)$ is a graph manifold. On the other hand for $k\geq 2$ and $m\geq 3$, the manifolds $M(k, m)$ are irreducible, Haken and atoroidal Hyperbolic 3-manifolds by \cite[Lemma 6]{MaR}. 
\end{remark}

\subsubsection{Generalized Takahashi manifolds}
Fix two positive integers $n, m$ and a collection of integers $\{p_{k,j}, q_{k,j}, r_{k,j}, s_{k,j}\}$ satisfying $gcd(p_{k,j}, q_{k,j}) = 1$, $gcd(r_{k,j}, s_{k,j}) = 1$ , and $p_{k,j}, r_{k,j} \geq 0$ for all $1\leq k\leq n$ and $1\leq j \leq m$. The generalized Takahashi manifold $T_{n,m}(\frac{p_{k,j}}{q_{k,j}}; \frac{r_{k,j}}{s_{k,j}})$ is the double branched cover of $\mathbb{S}^3$, branched over the closure of the braid appearing in Figure \ref{fig1} \cite[Theorem 3]{MV}, first defined in \cite[Section 2]{MV}.  We will denote the closure of this braid by $L_{n,m}(\frac{p_{k,j}}{q_{k,j}}; \frac{r_{k,j}}{s_{k,j}})$.

\begin{figure}[H]
	\centering
	\def\svgwidth{0.70\columnwidth}
	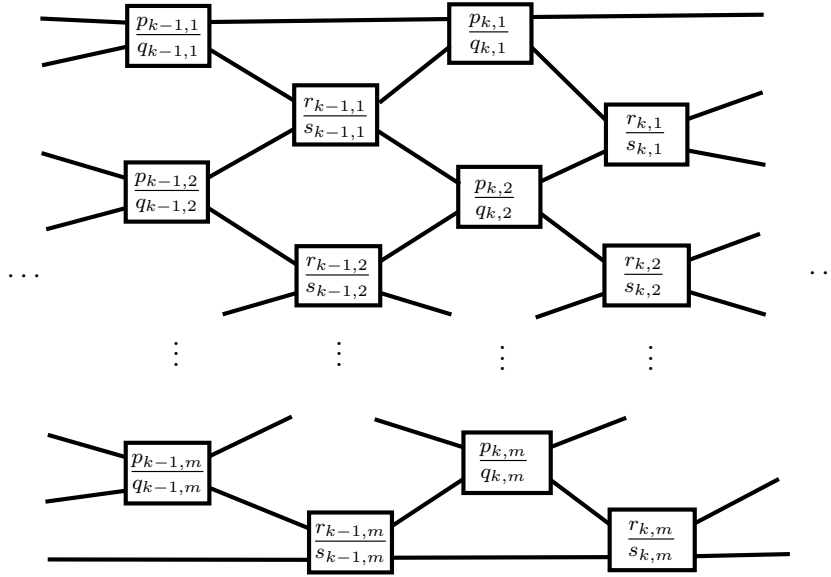
	\caption{The braid that defines the generalized Takahashi manifold $T_{n,m}(\frac{p_{k,j}}{q_{k,j}}; \frac{r_{k,j}}{s_{k,j}})$.  The fraction used to label each box determines the rational tangle used in that box to create $L_{n,m}(\frac{p_{k,j}}{q_{k,j}}; \frac{r_{k,j}}{s_{k,j}})$.}
	\label{fig1}
\end{figure}

This family of manifolds contains many well know $3$-manifolds, such as all $n$-fold cyclic branched covers of $2$-bridge knots, the Weeks manifold, and some graph manifolds. When the $p_{k,j}=p_j$, $q_{k,j}=q_j$, $r_{k,j}=r_j$ and $s_{k,j}=s_j$ for all $1\leq k\leq n$ and $1\leq j \leq m$, the manifold $T_{n,m}(\frac{p_{j}}{q_{j}}; \frac{r_{j}}{s_{j}})$ is called generalized periodic Takahashi manifold, correspondingly it is the double branched cover of the link $L_{n,m}(\frac{p_{j}}{q_{j}}; \frac{r_{j}}{s_{j}})$.

Each generalized periodic Takahashi manifold can also be viewed as a cyclic branched cover over a knot in a connected sum of lens spaces.

\begin{theorem}{\rm \cite[Theorem 6]{MV}}\label{T1} The generalized periodic Takahashi manifold $T_{n,m}(\frac{p_{j}}{q_{j}}; \frac{r_{j}}{s_{j}})$ is the $n$-fold cyclic branched cover of the connected sum of $2m$ lens spaces $$L(p_1,q_1)\# L(r_1, s_1)\# \cdots \# L(p_m,q_m)\# L(r_m, s_m)$$ branched over a knot which does not depend on $n$.
\end{theorem}

We use this as follows.

\begin{theorem}
The fundamental group of any generalized periodic Takahashi manifold \\ $T_{n,m}(\frac{p_j}{q_j}; \frac{r_j}{s_j})$ is circularly orderable if the set $\{p_j, r_j \mid \; 1\leq j \leq m\}$ contains at least two elements different from $1$, and the link $L_{n,m}(\frac{p_{j}}{q_{j}}; \frac{r_{j}}{s_{j}})$ is prime.
\end{theorem}
\begin{proof} Since the link $L_{n,m}(\frac{p_{j}}{q_{j}}; \frac{r_{j}}{s_{j}})$ is prime, $T_{n,m}(\frac{p_j}{q_j}; \frac{r_j}{s_j})$ is irreducible by the equivariant sphere theorem \cite{MSY} and the positive answer of the Smith conjecture \cite{MB}.
By Theorem \ref{T1} and \cite[Lemma 2.11]{GL}, there exists a surjective homomorphism from the fundamental group of the generalized periodic Takahashi manifold $\pi_1(T_{n,m}(\frac{p_{j}}{q_{j}}; \frac{r_{j}}{s_{j}}))$ to the free product
$\Z_{p_1}\ast \Z_{r_1}\ast \cdots \ast \Z_{p_m} \ast \Z_{r_m}$. Therefore, if the set $\{p_j, r_j \mid \; 1\leq j \leq m\}$ contains at least two elements different from $1$, this free product is infinite, and hence by Proposition \ref{COS}, $\pi_1(T_{n,m}(\frac{p_{j}}{q_{j}}; \frac{r_{j}}{s_{j}}))$ is circularly orderable.
\end{proof}
\begin{remark}
{\rm 

     The family of generalized Fibonacci manifolds is a subfamily of the family of generalized periodic Takahashi manifolds.
   
}
\end{remark}

\begin{question}
Is it possible to characterize the knots $K \subset \mathbb{S}^3$ for which $\pi_1(\Sigma_2(K))$ is circularly orderable?
\end{question}

\subsection{Dehn surgery}

In this brief section, we point out that circular orderability of manifolds arising from Dehn surgery on a knot in an integer homology $3$-sphere has already appeared in the literature under a different guise, from which we already observe different behaviour than left-orderability with respect to Dehn surgery.  

Recall that for a knot $K$ in an irreducible integer homology $3$-sphere $M$, the result of $p/q$ Dehn surgery on $M$ is denoted by $M_{p/q}(K)$.  The L-space conjecture predicts that for a knot $K$ in $\mathbb{S}^3$, if $\pi_1(\mathbb{S}^3_{p/q}(K))$ is non-left-orderable for some $p/q >0$, then in fact $\pi_1(\mathbb{S}^3_{p/q}(K))$ is non-left-orderable precisely when $p/q \geq 2g(K) -1$, where $g(K)$ is the genus of $K$ (\cite[Proposition 2.1]{OSz2} and \cite[Theorem 1]{R}) 

 To contrast this with circular orderability, we first observe that in light of Theorem \ref{LO_covers}:

\begin{proposition}
Suppose that $M$ is compact, connected, $\P^2$-irreducible $3$-manifold and that $H_1(M; \mathbb{Z})$ is cyclic.  Then $\pi_1(M)$ is circularly orderable if and only if the universal abelian cover of $M$ has left-orderable fundamental group.
\end{proposition}

Consequently, for an irreducible integer homology $3$-sphere $M$, since we have $H_1(M_{p/q}(K) ; \mathbb{Z} ) \cong \mathbb{Z} / p \mathbb{Z}$, we precisely need to investigate the universal abelian covers of these manifolds in order to know whether or not their fundamental groups are circularly orderable.  This is precisely what is done in \cite{BH}.

When $K$ is fibred, we use $h$ to denote its monodromy and $c(h)$ the fractional Dehn twist coefficient of $h$, see \cite[Section 4]{BH} for details. 

\begin{theorem}[\cite{BH}]
\label{boyerhu}
Suppose that $K$ is a fibred hyperbolic knot in an irreducible integer homology $3$-sphere $M$.  Given coprime $p, q$ with $p \geq 1$, the universal abelian cover of $M_{p/q}(K)$ has left-orderable fundamental group whenever 
\begin{enumerate}
\item $pc(h) \in \mathbb{Z}$ and $q \neq pc(h)$, and
\item $pc(h) \notin \mathbb{Z}$ and $q \notin \{ \lfloor pc(h) \rfloor, \lfloor pc(h) \rfloor + 1 \}$.
\end{enumerate}
\end{theorem}

Consequently, we observe that for any fibred knot $K$ in an irreducible integer homology $3$-sphere $M$, the result of $p/q$ surgery is a manifold with circularly orderable fundamental group whenever the surgery coefficient $p/q$ satisfies either condition (1) or (2) of Theorem \ref{boyerhu}.   

\begin{question}
Fix a knot $K$ in an irreducible integer homology $3$-sphere $M$.  Is it true that the set $\{ p/q  \mid M_{p/q}(K) \mbox{ is not circularly orderable} \}$ is always finite?
\end{question}


\begin{thebibliography}{alpha}

\bibitem[Ba]{Ba} 
I. Ba,
\textit{L-spaces, left-orderability and two-bridge knots}, J. of Knot Theory and Its Ram., {\bf 28}, (2019), 1950019, 1--38.
\bibitem[BaC]{BaC}
 I. Ba and A. Clay,
\textit{The space of circular orderings and semiconjugacy}, Journal of Algebra, \textbf{586}, (2021), 582-606.
\bibitem[BS]{BS} 
 H. Baik and E. Samperton,
\textit{Space of invariant circular orders of groups}, Groups Geom. Dyn., {\bf 12}, (2018), 721--763.
\bibitem[BHMNOV]{BHMNOV}
W. Ballinger, C. Ching-Yun Hsu, W. Mackey, Y. Ni, T. Ochse and F. Vafaee,
\textit{The prism manifold realization problem}, Alg. and Geom. Topol., {\bf 20}, (2020) 757--816.
\bibitem[BCG]{BCG}
J. Bell, A. Clay and T. Ghaswala,
\textit{Promoting circular-orderability to left-orderability}, Annales de l'institut Fourier, {\bf 71}, (2021), 175-201.
\bibitem[BG]{BG}
V. V. Bludov and A.M.W. Glass,
\textit{Word problems, embeddings, and free products of right-ordered groups with amalgamated
subgroup}, Proc. Lond. Math. Soc., {\bf 99}, (2009), 585--608.
\bibitem[BBG]{BBG}
M. Boileau, S. Boyer and C. McA. Gordon, 
\textit{ Branched covers of quasipositive links and L-spaces}, J. of Topology, {\bf 12}, (2019), 536--576.
\bibitem[BC]{BC}
S. Boyer and A. Clay,
 \textit{Foliations, orders, representations, L-spaces and graph manifolds}, Adv. Math., {\bf 330}, (2017), 159--234.
\bibitem[BC2]{BC2}
 S. Boyer and A. Clay,
 \textit{Order-detection of slopes on the boundaries of knot manifolds}, to appear in Groups, Geometry, Dynamics, (2022), arXiv:2206.00848.
\bibitem[BGW]{BGW}
S. Boyer, C. McA. Gordon and L. Watson,
\textit{On L-spaces and left-orderable fundamental groups}, Math. Ann., {\bf 356}, (2013), 1213--1245.
\bibitem[BH]{BH}
S. Boyer and Y. Hu,
\textit{Taut foliations in branched cyclic covers and left-orderable groups}, Trans. Amer. Math. Soc., {\bf 372}, (2019), 7921--7957.
\bibitem[BGH]{BGH}
S. Boyer, C. McA. Gordon and Y. Hu,
\textit{Slope
detection and toroidal $3$-manifolds}, preprint, (2021), arXiv:2106.14378.
\bibitem[BRW]{BRW} 
S. Boyer, D. Rolfsen and B. Wiest,
\textit{Orderable 3-manifold groups}, Ann. Inst. Fourier, {\bf 55}, (2005), 243--288.

\bibitem[Cal]{Cal} 
D. Calegari, 
\textit{Circular groups, Planar groups and Euler class}, Geom. and Topol. Monographs, {\bf 7}, (2004), 431--491.
\bibitem[CaD]{CaD} 
D. Calegari and N. M. Dunfield, 
\textit{Laminations and groups of homeomorphisms of the circle}, Invent. Math., {\bf 152}, (2003), 149--207.

\bibitem[CLW]{CLW}
A. Clay, T. Lidman and L. Watson, \textit{Graph manifolds, left-orderability and amalgamation},
Algebr. Geom. Topol., {\bf 13}, (2013), 2347--2368.
\bibitem[CG]{CG} 
 A. Clay and T. Ghaswala,
\textit{Free products of circularly ordered groups with amalgamated subgroup}, J. London Math. Soc., {\bf 2}, (2019), 1--29.
\bibitem[CG2]{CG2} 
 A. Clay and T. Ghaswala,
\textit{Circular ordering direct products and the obstruction to left-orderability},  Pacific J. Math., {\bf 312}, (2021), 401--419.
\bibitem[CG3]{CG3} 
 A. Clay and T. Ghaswala,
\textit{Cofinal Elements and Fractional Dehn Twist Coefficients}, Int. Math. Res. Not., (2023), rnad200, https://doi.org/10.1093/imrn/rnad200.

\bibitem[CMR]{CMR}
A. Clay and K. Mann, and C. Rivas,
\textit{On the number of circular orders on a group}, J. Algebra, {\bf 504}, (2018), 336--363.

\bibitem[CW]{CW} 
A. Clay and L. Watson,
\textit{Left-Orderable Fundamental Groups and Dehn Surgery}, Int. Math. Res. Not.  IMRN, {\bf 2013},  (2013), 2862--2890.
 \bibitem[CuD]{CuD}
M. Culler and N. M. Dunfield, \textit{Orderability and Dehn filling}, Geom. Topol., {\bf 22}, (2018), 1405--1457.
 \bibitem[DPT]{DPT}
M. Dabkowski and J. Przytycki, and A. Togha,
\textit{Non-Left-Orderable 3-Manifold Groups}, Canad. Math. Bull., {\bf 48}, (2005), 32--40.
\bibitem[DSh]{DSh}
R. J. Daverman and R. B. Sher,
\textit{Handbook of Geometric Topology}, Elsevier, (2002).

\bibitem[Ep]{Ep}
D. Epstein,
\textit{Projective planes in $3$-manifolds}, Proc. London Math. Soc., {\bf 11}, (1960), 1387--1427.


\bibitem[Ga]{Ga}
X. Gao,
\textit{Orderability of Homology Spheres Obtained by Dehn Filling}, Math. Res. Lett., {\bf 29}, (2022), 469--484.

\bibitem[Gh]{Gh}
E. Ghys,
\textit{Groups acting on the circle}, Enseign. Math., {\bf 47},  (2001), 329--407.

\bibitem[GL]{GL} 
C. McA. Gordon and T. Lidman,
\textit{Taut foliations, left-orderability, and cyclic branched covers}, Acta Math. Vietnam, {\bf 39}, (2014), 599--635.
\bibitem[HT1]{HT1}
R. Hakamata and M. Teragaito, \textit{Left-orderable fundamental groups and Dehn surgery on genus one 
2-bridge knots}, Algebr. Geom. Topol., {\bf 14}, (2014), 2125--2148.
\bibitem[HT2]{HT2}
R. Hakamata and M. Teragaito,
\textit{Left-orderable fundamental group and Dehn surgery on the knot $5_2$}, Canad. Math. Bull., {\bf 57}, (2014), 310--317.
\bibitem[HRRW]{HRRW}
J. Hanselman, J. Rasmussen, S. D. Rasmussen and L. Watson,
\textit{L-spaces, taut foliations, and graph manifolds}, Compositio Mathematica, {\bf 156}, (2020), 604--612.
\bibitem[Hei1]{Hei1}
W. Heil, \textit{Almost sufficiently large Seifert fiber spaces}, Michigan Math. J., {\bf 20}, (1973), 217--223.
\bibitem[Hei2]{Hei2}
W. Heil, \textit{Elementary surgery on Seifert fiber spaces}, Yokohama Math. J., {\bf 22}, (1974), 135--139.

\bibitem[Hu]{Hu} 
 Y. Hu, 
\textit{The left-orderability and the cyclic branched coverings}, Alg. and Geom. Topol., {\bf 15}, (2015), 399--413. 
\bibitem[Jab]{Jab}
S. Jablan, \textit{Tables of quasi-alternating knots with at most 12 crossings}, preprint, (2014), arXiv:1404.4965.
\bibitem[Ja]{Ja} 
W. Jaco, 
\textit{Lectures on three-manifold topology}, CBMS Regional Conf. Ser.
Math., {\bf 43}, (1980).

\bibitem[JR]{JR}
J. Remigio-Ju\'arez and Y. Rieck,
\textit{The Link Volumes of some prism manifolds}, Alg. and Geom. Topol., {\bf 12}, (2012), 1649--1665.

\bibitem[Ju]{Ju} 
A. Juh\'{a}sz, 
\textit{A survey of Heegaard Floer homology}, In New ideas in low dimensional topology, {\bf 56} of Ser. Knots Everything, pages 237--296. (2015).


\bibitem[MaR]{MaR} 
C. Maclachlan and A. Reid,
\textit{Generalised Fibonacci Manifolds}, J. of Transformation Groups, {\bf 2}, (1997), 165--182.

\bibitem[MeV]{MeV}
 A. Mednykh, A. Vesnin, \textit{Visualization of the isometry group action on the Fomenko-Matveev-Weeks manifold}, J. Lie Theory, {\bf 8}, (1998), 51--66.
\bibitem[MSY]{MSY}
W. Meeks, L. Simon, and S. T. Yau, 
\textit{Embedded minimal surfaces, exotic spheres, and manifolds
with positive Ricci curvature}, Ann. Math., {\bf 116}, (1982), 621--659.


\bibitem[MB]{MB}
J. W. Morgan and H. Bass, 
\textit{The Smith Conjecture}, Pure and Applied Math., {\bf 112}, Academic Press, 1984.
\bibitem[MT]{MT}
 K. Motegi and M. Teragaito, \textit{Generalized torsion elements and bi-orderability of 3-manifold groups}, Canad. Math. Bull., {\bf 60}, (2017), 830--844.

\bibitem[MV]{MV} 
M. Mulazzani and A. Vesnin,
\textit{Generalized Takahashi manifolds}, Osaka J. Math.  
{\bf 39}, (2002), 705--721. 

\bibitem[OSz]{OSz} 
 P. Ozsv\'ath and Z. Szab\'o,
\textit{On the Heegaard Floer homolgy of branched double-covers}, Adv. Math., 
{\bf 194} (2005), 1--33.


\bibitem[OSz2]{OSz2} 
 P. Ozsv\'ath and Z. Szab\'o,
\textit{On knot Floer homology and lens space surgeries}, Topology, 
{\bf 44} (2005), 1281--1300.

\bibitem[Pe]{Pe}
H. Pedersen, \textit{ Splice diagram determining singularity links and universal abelian covers}, Geom. Dedicata, {\bf 150}, (2011), 75--104.
\bibitem[Plo]{Plo}
S. Plotnick,
\textit{ Finite group actions and nonseparating 2-spheres}, Proc. Amer. Math. Soc., {\bf 90}, (1984), 430--432.

\bibitem[R]{R}
J. Rasmussen,
\textit{Lens space surgeries and {L}-space homology spheres}, (2016), preprint, arXiv:0710.2531.


\bibitem[Scot1]{Scot1}
  P. Scott, 
  \textit{The geometries of 3-manifolds}, Bull. London Math. Soc, {\bf 15}, (1983), 401--487.
 
\bibitem[Tr]{Tr} 
A. T. Tran, 
\textit{On left-orderability and cyclic branched coverings}, J. Math. Soc. Japan, {\bf 67}, (2015), 1169--1178. 
\bibitem[Tu]{Tu}
H. Turner,
\textit{Left-orderability, branched covers and double twist knots}, Proc. Amer. Math. Soc., {\bf 149}, (2021), 1343--1358.
\bibitem[V]{V} 
A. A. Vinogradov, \textit{On the free product of ordered groups}, Mat. Sb., {\bf 67}, (1949),
163--168.
\bibitem[Zel]{Zel}
 S. D. Zeleva,
 \textit{Cyclically ordered groups}, Sibirsk. Mat.  $\check{Z.}$, 
 {\bf 17}, (1976), 1046--1051, 1197.
  \end{thebibliography}
\end{document}